\documentclass{article}

\usepackage{amsmath}
\usepackage{amsfonts}
\usepackage{amsthm}
\usepackage{ifthen}
\usepackage{mathrsfs}
\usepackage{mathtools}
\usepackage{hyperref}
\usepackage{MnSymbol} 
\usepackage{verbatim} 
\usepackage[shortlabels]{enumitem}

\newtheorem{theorem}{Theorem}[section]
\newtheorem{lemma}[theorem]{Lemma}
\newtheorem{proposition}[theorem]{Proposition}
\newtheorem{corollary}[theorem]{Corollary}
\theoremstyle{definition}
\newtheorem{remark}[theorem]{Remark}
\newtheorem{definition}[theorem]{Definition}

\newtheorem*{conjecture*}{Conjecture}

\newcommand{\inr}[2]{\langle #1, #2 \rangle}

\newcommand{\R}{\mathbb{R}}
\newcommand{\E}{\mathbb{E}}

\newcommand{\n}{\mathbf{n}}
\newcommand{\tang}{\mathbf{t}}
\newcommand{\calH}{\mathcal{H}}

\newcommand{\cyclic}{\mathcal{C}}

\newcommand{\per}{P_\gamma}
\newcommand{\Id}{\mathrm{Id}}

\newcommand{\simplex}{\Delta}
\renewcommand{\div}{\mathrm{div}}
\newcommand{\isop}{I^{(1)}_m}
\newcommand{\pot}{W}

\newcommand{\modelV}{V}
\newcommand{\navg}{\overline{\n}}
\renewcommand{\H}{\mathcal{H}}
\newcommand{\Trips}{\mathcal{T}}

\newcommand{\norm}[1]{\left\Vert#1\right\Vert}

\newcommand{\abs}[1]{\left\vert#1\right\vert}
\newcommand{\set}[1]{\left\{#1\right\}}
\newcommand{\brac}[1]{\left(#1\right)}
\newcommand{\scalar}[1]{\left \langle #1 \right \rangle}

\newcommand{\Real}{\mathbb{R}}

\newcommand{\B}{\mathcal{B}}
\newcommand{\eps}{\epsilon}

\renewcommand{\div}{\text{div}}
\renewcommand{\S}{\mathbb{S}}
\newcommand{\G}{\mathbb{G}}
\renewcommand{\v}{\textrm{Vol}}
\newcommand{\F}{\mathcal{F}}

\DeclareMathOperator{\interior}{int}

\DeclareMathOperator{\tr}{tr}

\DeclareMathOperator{\cl}{cl}

\newcommand{\diffII}[3]{\ifthenelse{\equal{#2}{#3}}
{\frac{d^2 #1}{d #2^2}}
{\frac{d^2 #1}{d #2 d #3}}
}
\newcommand{\diffIIat}[4]{\left.\diffII{#1}{#2}{#3}\right|_{#4}}

\newcommand{\diffat}[3]{\left.\frac{d #1}{d #2}\right|_{#3}}
\newcommand{\pdiff}[2]{\frac{\partial #1}{\partial #2}}
\newcommand{\pdiffat}[3]{\left.\frac{\partial #1}{\partial #2}\right|_{#3}}

\begin{document}

\title{The Gaussian Double-Bubble Conjecture}

\author{Emanuel Milman\textsuperscript{1} and Joe Neeman\textsuperscript{2}}

\maketitle

\textbf{For publication purposes, this manuscript has been merged with our subsequent one ``The Gaussian Multi-Bubble Conjecture" into a single paper entitled ``The Gaussian Double-Bubble and Multi-Bubble Conjectures", which now supersedes both previous manuscripts -- see \href{https://arxiv.org/abs/1805.10961}{arXiv:1805.10961}}.

\abstract{We establish the Gaussian Double-Bubble Conjecture: 
the least Gaussian-weighted perimeter way to decompose $\R^n$ into three cells of prescribed (positive) Gaussian measure is to use a tripod-cluster, whose interfaces consist of three half-hyperplanes meeting along an $(n-2)$-dimensional plane at $120^{\circ}$ angles (forming a tripod or ``Y" shape in the plane). Moreover, we prove that tripod-clusters are the unique isoperimetric minimizers (up to null-sets).  
}

\footnotetext[1]{Department of Mathematics, Technion - Israel
Institute of Technology, Haifa 32000, Israel. Email: emilman@tx.technion.ac.il.}

\footnotetext[2]{Department of Mathematics, University of Texas at Austin. Email: joeneeman@gmail.com. \\
The research leading to these results is part of a project that has received funding from the European Research Council (ERC) under the European Union's Horizon 2020 research and innovation programme (grant agreement No 637851). This material is also based upon work supported by the National Science Foundation under Grant No. 1440140, while the authors were in residence at the Mathematical Sciences Research Institute in Berkeley, California, during the Fall semester of 2017.
}

\section{Introduction}

Let $\gamma = \gamma^n$ denote the standard Gaussian probability measure on Euclidean space $(\R^n,\abs{\cdot})$: \[
 \gamma^n := \frac{1}{(2\pi)^{\frac{n}{2}}} e^{-\frac{|x|^2}{2}} \, dx =: e^{-\pot(x)}\, dx.
\]
More generally, if $\H^{k}$ denotes the $k$-dimensional Hausdorff measure, let $\gamma^k$ denote its Gaussian-weighted counterpart: 
\[
\gamma^k := e^{-\pot(x)} \H^{k} .
\]
The Gaussian-weighted (Euclidean) perimeter of a Borel set $U \subset \R^n$ is defined as:
\[ \per(U) := \sup \left\{\int_U (\div X - \inr{\nabla \pot}{X}) \, d\gamma: X \in C_c^\infty(\R^n; T \R^n), |X| \le 1 \right\}.
\] For nice sets (e.g. open sets with piecewise smooth boundary), $\per(U)$ is known to agree with $\gamma^{n-1}(\partial U)$ (see, e.g.~\cite{MaggiBook}). The weighted perimeter $\per(U)$ has the advantage of being lower semi-continuous with respect to $L^1(\gamma)$ convergence, and thus fits well with the direct method of calculus-of-variations.

The classical Gaussian isoperimetric inequality, established independently by Sudakov--Tsirelson \cite{SudakovTsirelson} and Borell \cite{Borell-GaussianIsoperimetry} in 1975, asserts that among all Borel sets $U$ in $\R^n$ having prescribed Gaussian measure $\gamma(U) = v \in [0,1]$, halfplanes minimize Gaussian-weighted perimeter $\per(U)$ (see also \cite{EhrhardPhiConcavity, BakryLedoux, BobkovGaussianIsopInqViaCube, BobkovLocalizedProofOfGaussianIso,MorganManifoldsWithDensity}).
Later on, it was shown by Carlen--Kerce \cite{CarlenKerceEqualityInGaussianIsop} (see also \cite{EhrhardGaussianIsopEqualityCases,MorganManifoldsWithDensity, McGonagleRoss:15}), that up to $\gamma$-null sets, halfplanes are in fact the \emph{unique} minimizers for the Gaussian isoperimetric inequality. 

\medskip

In this work, we extend these classical results to the case of $3$-clusters. A \emph{$k$-cluster} $\Omega = (\Omega_1, \ldots, \Omega_k)$ is a $k$-tuple of Borel subsets $\Omega_i \subset \R^n$ called cells, such that $\set{\Omega_i}$ are pairwise disjoint, $\per(\Omega_i) <
\infty$ for each $i$, and  $\gamma(\R^n \setminus \bigcup_{i=1}^k \Omega_i) = 0$. Note that the cells are not required to be connected. The total Gaussian
perimeter of a cluster $\Omega$ is defined as:
\[
  \per(\Omega) := \frac 12 \sum_{i=1}^k \per(\Omega_i) .
\]
The Gaussian measure of a cluster is defined as:
\[
  \gamma(\Omega) := (\gamma(\Omega_1), \ldots, \gamma(\Omega_k)) \in \simplex^{(k-1)} ,
\]
where $\simplex^{(k-1)} := \{v \in \R^k: v_i \ge 0 ~,~ \sum_{i=1}^k v_i = 1\}$ denotes the $(k-1)$-dimensional simplex. The isoperimetric problem for $k$-clusters consists of identifying those clusters $\Omega$ of prescribed Gaussian measure $\gamma(\Omega) = v \in \simplex^{(k-1)}$ which minimize the total Gaussian perimeter $\per(\Omega)$. 

\smallskip
Note that easy properties of the perimeter ensure that for a $2$-cluster $\Omega$, $\per(\Omega) = \per(\Omega_1) = \per(\Omega_2)$, and so the case $k=2$ corresponds to the classical isoperimetric setup when testing the perimeter of a single set of prescribed measure. In analogy to the classical \emph{unweighted} isoperimetric inequality in Euclidean space $(\R^n,\abs{\cdot})$, in which the Euclidean ball minimizes (unweighted) perimeter among all sets of prescribed Lebesgue measure \cite{BuragoZalgallerBook,MaggiBook}, we will refer to the case $k=2$ as the ``single-bubble" case (with the bubble being $\Omega_1$ and having complement $\Omega_2$). Accordingly, the case $k=3$ is called the ``double-bubble" problem. See below for further motivation behind this terminology and related results. 

\smallskip
A natural conjecture is then the following:
\begin{conjecture*}[Gaussian Multi-Bubble Conjecture]
For all $k \leq n+1$, the least (Gaussian-weighted) perimeter way to decompose $\R^n$ into $k$ cells of prescribed (Gaussian) measure $v \in \interior \simplex^{(k-1)}$ is by using the Voronoi cells of $k$ equidistant points in $\R^n$. \end{conjecture*}
\noindent 
Recall that the Voronoi cells of $\set{x_1,\ldots,x_k} \subset \R^n$ are defined as:
\[
\Omega_i = \interior \set{ x \in \R^n : \min_{j=1,\ldots,k} \abs{x-x_j} = \abs{x-x_i} } \;\;\; i=1,\ldots, k ~ ,
\]
where $\interior$ denotes the interior operation (to obtain pairwise disjoint cells). Indeed, when $k=2$, the Voronoi cells are precisely halfplanes, and the single-bubble conjecture holds by the classical Gaussian isoperimetric inequality. When $k=3$ and $v \in \interior \Delta^{(2)}$, the conjectured isoperimetric minimizing clusters are tripods (also referred to as ``Y's", ``rotors" or ``propellers" in the literature), whose interfaces are three half-hyperplanes meeting along an $(n-2)$-dimensional plane at $120^\circ$ angles (forming a tripod or ``Y" shape in the plane). When $v \in \partial \Delta^{(2)}$, the 3-cluster has (at least) one null-cell, reducing to the $k=2$ case (and indeed two complementing half-planes may be seen as a degenerate tripod cluster, whose vertex is at infinity). 

\smallskip
Tripod clusters are the naturally conjectured minimizers in the double-bubble case, as their interfaces have constant Gaussian mean-curvature (being flat), and meet at $120^{\circ}$ angles, both of which are necessary conditions for any extremizer of (Gaussian) perimeter under a (Gaussian) measure constraint -- see Section \ref{sec:first-order}. Note that tripod clusters are also known to be the global extremizers in other optimization problems, such as the problem of maximizing the sum of the squared lengths of the Gaussian moments of the cells in $\R^3$ (see Heilman--Jagannath--Naor \cite{HeilmanJagannathNaor-PropellerInR3}). On the other hand, for the Gaussian noise-stability question, halfplanes are known to maximize noise-stability in the single-bubble case \cite{Borell-GaussianNoiseStability, MosselNeeman-GaussianNoiseStability}, but tripod-clusters are known to not always maximize noise-stability in the double-bubble case \cite{HeilmanMosselNeeman-GaussianNoiseStability}.

\smallskip
Our main result in this work is the following:

\begin{theorem}[Gaussian Double-Bubble Theorem] \label{thm:main1} \label{thm:main-I-I_m}
The Gaussian Double-Bubble Conjecture (case $k=3$) holds true for all $n \geq 2$. 
\end{theorem}

In addition, we resolve the uniqueness question:
\begin{theorem}[Uniqueness of Minimizing Clusters] \label{thm:main-uniqueness}
Up to $\gamma^n$-null sets, tripod clusters are the \emph{unique} minimizers of Gaussian perimeter among all clusters of prescribed Gaussian measure $v \in \interior \Delta^{(2)}$. 
\end{theorem}

\subsection{Previously Known and Related Results}

The Gaussian Multi-Bubble Conjecture is known to experts. Presumably, its origins may be traced to an analogous problem of J.~Sullivan from 1995 in the \emph{unweighted} Euclidean setting \cite[Problem 2]{OpenProblemsInSoapBubbles96}, where the conjectured uniquely minimizing cluster (up to null-sets) is a standard $k$-bubble -- spherical caps bounding connected cells $\set{\Omega_i}_{i=1}^{k}$ which are obtained by taking the Voronoi cells of $k$ equidistant points in $\S^{n} \subset \R^{n+1}$, and applying all stereographic projections to $\R^n$.  

To put our results into appropriate context, let us go over some related results in three settings: the unweighted Euclidean setting $\R^n$, on the $n$-sphere $\S^n$ endowed with its canonical Riemannian metric and measure (normalized for convenience to be a probability measure), 
and finally in our Gaussian-weighted Euclidean setting $\G^n$. Further results may be found in F.~Morgan's excellent book \cite[Chapters 13,14,18,19]{MorganBook5Ed}. 

\begin{itemize}
\item The classical isoperimetric inequality in the unweighted Euclidean setting, going back (at least in dimension $n=2$ and perhaps $n=3$) to the ancient Greeks, and first proved rigorously by Schwarz in $\R^n$ (see \cite[Chapter 13.2]{MorganBook5Ed}, \cite[Subsection 10.4]{BuragoZalgallerBook} and the references therein), states that the Euclidean ball is a single-bubble isoperimetric minimizer in $\R^n$. It was shown by DeGiorgi \cite[Theorem 14.3.1]{BuragoZalgallerBook} that up to null-sets, it is uniquely minimizing perimeter. Long believed to be true, but appearing explicitly as a conjecture in an undergraduate thesis by J.~Foisy in 1991 \cite{Foisy-UGThesis}, 
the double-bubble case was considered in the 1990's by various authors \cite{SMALL93, HHS95}, culminating in the work of Hutchings--Morgan--Ritor\'e--Ros \cite{DoubleBubbleInR3}, who proved that up to null-sets, the standard double-bubble is uniquely perimeter minimizing in $\R^3$; this was later extended to $\R^n$ in \cite{SMALL03,Reichardt-DoubleBubbleInRn}. That the standard triple-bubble is uniquely perimeter minimizing in $\R^2$ was proved by Wichiramala in \cite{Wichiramala-TripleBubbleInR2}. 

\item  
On $\S^n$, it was shown by P.~L\'evy \cite{LevyIsopInqOnSphere} and Schmidt \cite{SchmidtIsopOnModelSpaces} that geodesic balls are single-bubble isoperimetric minimizers. Uniqueness up to null-sets was established by DeGiorgi \cite[Theorem 14.3.1]{BuragoZalgallerBook}. As in the unweighted Euclidean setting, the multi-bubble conjecture on $\S^n$ asserts that the uniquely perimeter minimizing $k$-cluster  (up to null-sets) is a standard $k$-bubble: spherical caps bounding connected cells $\set{\Omega_i}_{i=1}^{k}$ and having incidence structure as in the Euclidean and Gaussian cases, already described above. The double-bubble conjecture was resolved on $\S^2$ by Masters \cite{Masters-DoubleBubbleInS2}, but on $\S^n$ for $n\geq 3$ only partial results are known \cite{CottonFreeman-DoubleBubbleInSandH, CorneliHoffmanEtAl-DoubleBubbleIn3D,CorneliCorwinEtAl-DoubleBubbleInSandG}.  
In particular, we mention a result by Corneli-et-al \cite{CorneliCorwinEtAl-DoubleBubbleInSandG}, which confirms the double-bubble conjecture on $\S^n$ for all $n \geq 3$ when the prescribed measure $v \in \simplex^{(2)}$ satisfies $\max_{i} \abs{v_i - 1/3} \leq 0.04$. Their proof employs a result of Cotton--Freeman \cite{CottonFreeman-DoubleBubbleInSandH} stating that if the minimizing cluster's cells are known to be connected, then it must be the standard double-bubble. 
\item We finally arrive to the Gaussian setting $\G^n$ which we consider in this work. As already mentioned, halfplanes are the unique single-bubble isoperimetric minimizers (up to null-sets). The original proofs by Sudakov--Tsirelson and independently Borell both deduced the single-bubble isoperimetric inequality on $\G^n$ from the one on $\S^N$, exploiting the classical fact that the projection onto a fixed $n$-dimensional subspace of the uniform measure on a rescaled sphere $\sqrt{N} \S^N$, converges to the Gaussian measure $\gamma^n$ as $N \rightarrow \infty$. Building upon this idea, it was shown by Corneli-et-al \cite{CorneliCorwinEtAl-DoubleBubbleInSandG} that verification of the double-bubble conjecture on $\S^N$ for a sequence of $N$'s tending to $\infty$ and a fixed $v \in \interior \simplex^{(2)}$ will verify the double-bubble conjecture on $\G^n$ for the same $v$ and for all $n \geq 2$. As a consequence, they confirmed the double-bubble conjecture on $\G^n$ for all $n \geq 2$ when the prescribed measure $v \in \simplex^{(2)}$ satisfies $\max_{i} \abs{v_i - 1/3} \leq 0.04$. Note that this approximation argument precludes any attempts to establish uniqueness in the double-bubble conjecture, and to the best of our knowledge, no uniqueness in the double-bubble conjecture on $\G^n$ was known for any $v \in \interior \simplex^{(2)}$ prior to our Theorem \ref{thm:main-uniqueness}.
\end{itemize}

\medskip

An essential ingredient in the proofs of the double-bubble results in $\R^n$ (and in many results on $\S^n$), is a symmetrization argument due to B.~White (see Foisy \cite{Foisy-UGThesis} and Hutchings \cite{Hutchings-StructureOfDoubleBubbles}). However, it is not clear to us what type of Gaussian symmetrization will produce a tripod cluster (for a symmetrization which works in the single-bubble case, see Ehrhard \cite{EhrhardPhiConcavity}). A second essential ingredient in the above results is Hutchings' theory \cite{Hutchings-StructureOfDoubleBubbles} of bounds on the number of connected components comprising each cell of a minimizing cluster. 

\smallskip
In contrast, we do not use any of these ingredients in our approach. Furthermore, contrary to previous approaches, we do not directly obtain a lower bound on the perimeter of a cluster having prescribed measure $v \in \simplex^{(2)}$, nor do we identify the minimizing clusters by ruling out competitors. Our approach is based on obtaining a matrix-valued partial-differential inequality on the associated isoperimetric profile, concluding the desired lower bound 
in one fell swoop for \emph{all} $v \in \simplex^{(2)}$ by an application of the maximum-principle.

\subsection{Outline of Proof}

Let $I^{(k-1)} : \simplex^{(k-1)} \to \R_+$ denote the Gaussian isoperimetric profile for $k$-clusters, defined as:
\[
  I^{(k-1)}(v) := \inf\{\per(\Omega): \text{$\Omega$ is a $k$-cluster with $\gamma(\Omega) = v$}\}.
\]
Our goal will be to show that $I^{(2)} = I^{(2)}_m$ on $\interior \simplex^{(2)}$, where $I^{(2)}_m : \interior \simplex^{(2)} \to \R_+$ denotes the Gaussian  double-bubble \emph{model} profile:
\[
  I^{(2)}_m(v) := \inf\{\per(\Omega): \text{$\Omega$ is a tripod-cluster with $\gamma(\Omega) = v$}\}.
\]
Note that for any $v \in \interior \simplex^{(2)}$, there exists a tripod-cluster $\Omega^m$ with $\gamma(\Omega^m) = v$. Indeed, by the product structure of the Gaussian measure, it is enough to establish this in the plane, and after fixing the orientation of the tripod, we actually show using a topological argument that there exists a bijection between the vertex of the tripod in $\R^2$ and $\interior \simplex^{(2)}$.

To establish $I^{(2)} = I^{(2)}_m$, let us draw motivation from the single-bubble case. By identifying $\simplex^{(1)}$ with $[0,1]$ using the map $\simplex^{(1)} \ni (v_1,v_2) \mapsto v_1 \in [0,1]$, we will think of the single-bubble profile $I^{(1)}$ as defined on $[0,1]$. The single-bubble Gaussian isoperimetric inequality asserts that $I^{(1)} = I^{(1)}_m$, where $I^{(1)}_m : [0,1] \rightarrow \Real_+$ denotes the single-bubble \emph{model} profile:
\[
I^{(1)}_m(v) := \inf\{\per(U) : \text{$U$ is a halfplane with $\gamma(U) = v$} \} . 
\]
The product structure of the Gaussian measure implies that $I^{(1)}_m$ may be calculated in dimension one, readily yielding:
\[
I^{(1)}_m(v) = \varphi \circ \Phi^{-1}(v),
\]
where $\varphi(x) = (2\pi)^{-1/2} e^{-x^2/2}$ is the one-dimensional Gaussian density, and $\Phi(x) = \int_{-\infty}^x \varphi(y)\, dy$ is its cumulative distribution function. It is well-known and immediate to check that $I^{(1)}_m$ satisfies the following ODE:
\begin{equation} \label{eq:intro-ODE}
(I^{(1)}_m)^{\prime\prime} = -\frac{1}{I^{(1)}_m} \text{ on $[0,1]$.} 
\end{equation}

Our starting observation is that $I^{(2)}_m$ satisfies a similar \emph{matrix-valued} differential equality on $\simplex^{(2)}$. Let $E$ denote the tangent space to $\simplex^{(2)}$, which we identify with $\{x \in \R^3: \sum_{i=1}^3 x_i = 0\}$. Given $A = (A_{12},A_{23},A_{13})$ with $A_{ij} \geq 0$, we introduce the following $3 \times 3$ positive semi-definite matrix:
\[
L_{A} := \sum_{1 \leq i < j \leq 3} A_{ij} (e_i - e_j) (e_i-e_j)^T = 
\brac{\begin{matrix}
    A_{12} + A_{13} & -A_{12} & - A_{13} \\
    -A_{12}                & A_{12} + A_{23} & -A_{23} \\
    -A_{13}                & -A_{23}  &  A_{13} + A_{23} 
\end{matrix}} \geq 0 .
\]
In fact, as a quadratic form on $E$, it is easy to see that $L_A$ is strictly positive-definite as soon as at least two $A_{ij}$'s are positive.

Given $v \in \interior \simplex^{(2)}$, let $A^m_{ij}(v) := \gamma^{n-1}(\partial \Omega^m_i \cap \partial \Omega^m_j) > 0$ denote the weighted areas of the interfaces of a tripod cluster $\Omega^m$ satisfying $\gamma(\Omega^m) = v$. A calculation then verifies that:
\begin{equation} \label{eq:intro-MDE}
\nabla^2 I^{(2)}_m(v) = -L_{A^m(v)}^{-1} \text{ on $\interior \simplex^{(2)}$,}
\end{equation}
where differentiation and inversion are both carried out on $E$. This constitutes the right extension of (\ref{eq:intro-ODE}) to the double-bubble setting. 

\medskip

To establish that $I^{(2)} = I^{(2)}_m$ on $\interior \simplex^{(2)}$, our idea is as follows. First, it is not hard to show that $I^{(2)}_m$ may be extended continuously to the entire $\simplex^{(2)}$ by setting $I^{(2)}_m(v) := I^{(1)}_m(\max_i v_i)$ for $v \in \partial \simplex^{(2)}$. We clearly have $I^{(2)} \leq I^{(2)}_m$ on $\simplex^{(2)}$, with equality on the boundary by the single-bubble Gaussian isoperimetric inequality (where $I^{(2)}(v) = I^{(1)}(\max_i v_i) = I^{(1)}_m(\max_i v_i)$). 

\smallskip

Assume for the sake of this sketch that $I^{(2)}$ is twice continuously differentiable on $\interior \simplex^{(2)}$. Given an isoperimetric minimizing cluster $\Omega$ with $\gamma(\Omega) = v \in \interior \simplex^{(2)}$, let $A_{ij}(v) := \gamma^{n-1}(\partial^*\Omega_i \cap \partial^* \Omega_j)$, $i < j$, denote the weighted areas of the cluster's interfaces, where $\partial^* U$ denotes the reduced boundary of a Borel set $U$ having finite perimeter (see Section \ref{sec:prelim}). Assume again for simplicity that $A_{ij}(v)$ are well-defined, that is, depend only on $v$. It is easy to see that at least two $A_{ij}(v)$ must be positive, and hence $L_{A(v)}$ is positive-definite on $E$. We will then show that the following matrix-valued differential inequality holds:
\begin{equation} \label{eq:intro-MDI}
\nabla^2 I^{(2)}(v) \leq -L_{A(v)}^{-1} \text{ on $\interior \simplex^{(2)}$,}
\end{equation}
in the positive semi-definite sense on $E$. Consequently:
\begin{equation} \label{eq:intro-PDE}
- \tr[ (\nabla^2 I^{(2)}(v))^{-1} ] \leq \tr(L_{A(v)}) = 2 \sum_{i<j} A_{ij}(v) = 2 I^{(2)}(v)  \text{ on $\interior \simplex^{(2)}$.}
\end{equation}
On the other hand, by (\ref{eq:intro-MDE}), we have equality above when $I^{(2)}$ and $A$ are replaced by $I^{(2)}_m$ and $A^m$, respectively. Since $I^{(2)} \leq I^{(2)}_m$ on $\simplex^{(2)}$ with equality on the boundary, an application of the maximum principle for the (fully non-linear) second-order elliptic PDE (\ref{eq:intro-PDE}) yields the desired $I^{(2)} = I^{(2)}_m$. 

\smallskip

The bulk of this work is thus aimed at establishing a rigorous version of (\ref{eq:intro-MDI}). To this end, we consider an isoperimetric minimizing cluster $\Omega$, and perturb it using a flow $F_t$ along a vector-field $X$. Since:
\[
 I^{(2)}(\gamma(F_t(\Omega))) \le \per(F_t(\Omega)) 
\]
with equality at $t=0$, we deduce (at least, conceptually) that the first variations must coincide and that the second variations must obey the inequality. 
Such an idea in the single-bubble case is not new, and was notably used by Sternberg--Zumbrun in \cite{SternbergZumbrun} to establish concavity of the isoperimetric profile for a convex domain in the unweighted Euclidean setting; see also \cite{MorganJohnson,Kuwert, RosIsoperimetryInCrystals, BayleRosales} for further extensions and applications, \cite{KleinerProofOfCartanHadamardIn3D} for a first-order comparison theorem on Cartan-Hadamard manifolds, and \cite{BayleThesis} where the resulting second-order ordinary differential inequality was integrated to recover the Gromov--L\'evy isoperimetric inequality \cite{GromovGeneralizationOfLevy}, \cite[Appendix C]{Gromov}. 
On a technical level, it is crucial for us to apply this to a \emph{non-compactly supported} vector-field $X$, and so we spend some time to revisit classical results regarding the first and second variations of (weighted) volume and perimeter. Contrary to the single-bubble setting, the interfaces $\Sigma_{ij} = \partial^*\Omega_i \cap \partial^* \Omega_j$ will meet each other at common boundaries, which may contribute to these variations. Fortunately, for the first variation of perimeter, these contributions cancel out thanks to the isoperimetric stationarity of the cluster (without requiring us to use delicate regularity results regarding the structure of the \emph{boundary} of interfaces -- see Remark \ref{rem:no-higher-regularity}). For the second variation, we are not so fortunate, and in general the boundary of the interfaces will have a contribution.

\smallskip

However, in simplifying our original argument for establishing (\ref{eq:intro-MDI}), we obtain an interesting dichotomy: if the minimizing cluster is effectively one-dimensional (a case we do not a-priori rule out), it is easy to obtain (\ref{eq:intro-MDI}) by a one-dimensional computation directly; otherwise, it turns out we obtain enough information from applying the above machinery to the constant vector-field $X \equiv w$ in all directions $w \in \R^n$, amounting to translation of the minimizing cluster. This simplifies considerably the resulting formulas for the second variation of (weighted) volume and perimeter, and allows us again to bypass the delicate regularity results mentioned before. Consequently, we only need to use the classical results from Geometric Measure Theory on the existence of isoperimetric minimizing clusters and regularity of their interfaces, due to Almgren \cite{AlmgrenMemoirs} (see also Maggi's excellent book \cite{MaggiBook}).

\smallskip

To establish the uniqueness of minimizers, we observe that all of our inequalities in the derivation of (\ref{eq:intro-MDI}) must have been equalities, and this already provides enough information for characterizing tripod-clusters. 

\medskip

The rest of this work is organized as follows. In Section \ref{sec:model}, we construct the model tripod-clusters and associated model isoperimetric profile, and establish their properties.  
In Section \ref{sec:prelim} we recall relevant definitions and provide some preliminaries for the ensuing calculations. In Section \ref{sec:first-order} we deduce first-order properties of isoperimetric minimizing clusters (such as stationarity) in the weighted unbounded setting, as well as second-order stability. In Section \ref{sec:second-order}, we calculate the second variations of measure and weighted-perimeter under translations. In Section \ref{sec:MDI} we obtain a rigorous version of the matrix-valued differential inequality (\ref{eq:intro-MDI}) in both cases of the above-mentioned dichotomy. In Section \ref{sec:proof}, we conclude the proof of Theorem \ref{thm:main1} by employing a maximum-principle. In Section \ref{sec:uniqueness}, we establish Theorem \ref{thm:main-uniqueness} on the uniqueness of the isoperimetric minimizing clusters. Finally, in Section \ref{sec:conclude}, we provide some concluding remarks on extensions and additional results which will appear in \cite{EMilmanNeeman-GaussianMultiBubbleConj}. For brevity of notation, we omit the superscript $^{(2)}$ in all subsequent references to $I^{(2)}$, $I^{(2)}_m$ and $\simplex^{(2)}$ in this work. In addition, we use  $\cyclic = \{(1,2),(2,3),(3,1) \}$  to denote the set of positively oriented pairs in $\{1, 2, 3\}$.

\medskip

\noindent 
\textbf{Acknowledgement.} We thank Francesco Maggi for pointing us towards a simpler version of the maximum principle than our original one,
Frank Morgan for many helpful references, and Brian White for informing us of the reference to the recent \cite{CES-RegularityOfMinimalSurfacesNearCones}. We also acknowledge the hospitality of MSRI where part of this work was conducted.

\section{Model Tripod Configurations} \label{sec:model}

In this section, we construct the model tripod clusters which are conjectured to be optimal on $\R^n$. It will be enough to construct them on $\R^2$, since by taking Cartesian product with $\Real^{n-2}$ and employing the product structure of the Gaussian measure, these clusters extend to $\Real^n$ for all $n \geq 2$.
Actually, it will be convenient to construct them on $E$, rather than on $\R^2$, 
where recall that $E$ denotes the tangent space to $\simplex^{(2)}$, which we identify with $\{x \in \R^3: \sum_{i=1}^3 x_i = 0\}$.
Consequently, in this section, let $\gamma = \gamma^2$ denote the standard two-dimensional Gaussian measure on $E$, and if $\Psi$ denotes its (smooth, Gaussian) density on $E$, 
set $\gamma^1 := \Psi \H^1$. 

\medskip

Define $\Omega^m_i = \interior \{x \in E: \max_j x_j = x_i\}$ (``$m$'' stands for ``model'').
For any $x \in E$, $x + \Omega^m = (x + \Omega^m_1, x + \Omega^m_2, x + \Omega^m_3)$ is a cluster, which we call a ``tripod"  cluster (also referred to as a ``Y" or ``rotor" in the literature). 
Let $\Sigma^m_{ij} := \partial \Omega^m_i \cap \partial \Omega^m_j$ denote the interfaces of $\Omega^m$. 
Observe that $x + \Sigma^m_{ij}$, the interfaces of $x + \Omega^m$, are flat, and meet at $120^\circ$ angles at the single point $x$. We denote:
\[
A^m_{ij}(x) := \gamma^1(x + \Sigma^m_{ij}) .
\]

\medskip

We begin with some preparatory lemmas:

\begin{lemma} \label{lem:model-sigma}
For all distinct $i,j,k \in \{ 1,2,3\}$, let $n_{ij} = (e_j - e_i)/\sqrt 2$ and $t_{ij} =  (e_i + e_j - 2 e_k) / \sqrt 6$. Then $n_{ij}$ and $t_{ij}$ form an orthonormal basis of $E$, and:
\begin{equation}   \label{eq:tripod-areas}
A^m_{ij}(x) = \varphi(\inr{x}{n_{ij}}) (1 - \Phi(\inr{x}{t_{ij}}))
\end{equation}
\end{lemma}
\begin{proof}
It is immediate to check that $\Sigma^m_{ij}$ can be parametrized (with unit speed) as $\{a t_{ij}: a \ge 0\}$, and that $t_{ij}$ and $n_{ij}$ form an orthonormal basis of $E$. 
Consequently, for any $a \in \R$, $|a t_{ij} + x|^2 = (a + \inr{x}{t_{ij}})^2 + \inr{x}{n_{ij}}^2$. Hence,
  \begin{align*}
    \gamma^1(x + \Sigma_{ij})
    &= \frac{1}{2\pi} \int_0^\infty e^{-|a t_{ij} + x|^2/2} \, da \notag \\
    &= \varphi(\inr{x}{n_{ij}}) \int_0^\infty \varphi(a + \inr{x}{t_{ij}})\, da \notag \\
    &= \varphi(\inr{x}{n_{ij}}) (1 - \Phi(\inr{x}{t_{ij}})).
  \end{align*}
\end{proof}

\begin{lemma} \label{lem:model-upper-lower}
For all $x \in E$ and distinct $i,j,k \in \{ 1,2,3\}$:
\begin{align}
\label{eq:model-upper} \gamma(x + \Omega^m_i) &\leq \min\brac{1 - \Phi(x_i) , 1 - \Phi\brac{\frac{x_i-x_j}{\sqrt{2}}} , 1 - \Phi\brac{\frac{x_i-x_k}{\sqrt{2}}}} ~,\\
\label{eq:model-lower} \gamma(x + \Omega^m_i) & \geq \brac{1 - \Phi\brac{\frac{x_i-x_j}{\sqrt{2}}}}\brac{ 1 - \Phi\brac{\frac{x_i-x_k}{\sqrt{2}}} } ~.
\end{align}
\end{lemma}
\begin{proof}
For the upper bound, note that:
\[
   x + \Omega^m_i = \interior \{z \in E : \max_j (z_j - x_j) = z_i - x_i  \} \subseteq \{z \in E : z_i - x_i > 0\},
\]
where the last inclusion follows since $y = z-x\in E$ implies that $\max_j y_j \ge \frac{1}{3} \sum_{j=1}^3 y_j = 0$. Similarly, for any $j \neq i$:
\[
  x + \Omega^m_i = \interior \{z \in E : \max_j (z_j - x_j) = z_i - x_i  \} \subseteq \{ z \in E : z_i - z_j  > x_i - x_j \} .
\]
It remains to note that if $Z$ is distributed according to $\gamma$ on $E$, then $Z_i$ and $(Z_i - Z_j) / \sqrt{2}$ are distributed according to the standard Gaussian measure on $\Real$, and (\ref{eq:model-upper}) follows. 

For the lower bound,  note that:
 \[
    x + \Omega^m_i = \{z: z_i - z_j > x_i - x_j\} \cap \{z: z_i - z_k > x_i - x_k \}.
 \]
  By the Gaussian FKG inequality~\cite{Pitt:82},
  \[
    \gamma(x + \Omega^m_i) \ge \gamma(\{z: z_i - z_j > x_i - x_j\}) \gamma(\{z: z_i - z_k > x_i - x_k\}) ,
   \]
   and (\ref{eq:model-lower}) is established. 
\end{proof}

\begin{lemma} \label{lem:V-diffeo}
 The map:
 \[
 E \ni x \mapsto \modelV(x) := \gamma(x + \Omega^m) \in \interior \simplex 
 \]
 is a diffeomorphism between $E$ and $\interior \simplex$. Its differential is given by:
 \begin{equation}\label{eq:DV}
      D \modelV = -\frac{1}{\sqrt 2}
      \brac{\begin{matrix}
            A^m_{12} + A^m_{31} & -A^m_{12} & -A^m_{31} \\
            -A^m_{12} & A^m_{12} + A^m_{23} & -A^m_{23} \\
            -A^m_{31} & -A^m_{23} & A^m_{23} + A^m_{31}
        \end{matrix}}
        = - \frac{1}{\sqrt 2} L_{A^m} .
  \end{equation}
\end{lemma}

\begin{proof}
    Clearly $\modelV(x)$ is $C^\infty$,
    since the Gaussian density is $C^\infty$ and all of its derivatives vanish rapidly
    at infinity.
    
    To see that $\modelV$ is injective, simply note that if $y \neq x$, then there exists $i \in \{1,2,3\}$ such that $y \in x + \Omega^m_i$, and 
    therefore $y + \Omega^m_i \subsetneq x + \Omega^m_i$ and hence $\gamma(y + \Omega^m_i) < \gamma(x + \Omega^m_i)$. 
   
    To show that $\modelV$ is surjective, fix $R$ and consider the open triangle $K_R = \{x \in E:
    \max_i x_i < R\}$.  The boundary of $K_R$ is made up of three line
    segments, each of the form
    $K_{R,i} := \{x \in E : x_i = R, x_j \le R, x_k \le R\}$. By (\ref{eq:model-upper}), we know that for any $x \in K_{R,i}$ we have $\gamma(x + \Omega_i) \le 1 - \Phi(R)$, i.e. that $\modelV_i(x) \le 1 - \Phi(R)$ on $K_{R,i}$.
    It follows that for any $v \in \simplex_R := \set{ u \in \simplex : \min_i u_i > 1 - \Phi(R)}$,
    as $x$ runs clockwise along $\partial K_R$, then $\modelV(x)$ runs clockwise around $v$.
    By Invariance of Domain Theorem \cite{Munkres-TopologyBook2ndEd}, $\modelV$ is a homeomorphism between the open triangle $K_R$ and $\modelV(K_R)$, and since $\modelV(\partial K_R) \cap \simplex_R = \emptyset$, it follows that $\modelV(K_R)$ must contain $\simplex_R$. 
        Since every $v \in \interior \simplex$ satisfies $v \in \simplex_R$ for some $R < \infty$, the surjectivity follows.
    
    Finally, we compute at $x \in E$:
    \[
      \nabla_v \modelV_i(x) = \int_{x + \Omega^m_i} \nabla_v e^{-\pot(x)} \, dx
      = \int_{x + \partial \Omega^m_i} \inr{v}{\n} e^{-\pot(x)}\, d\calH^{n-1},
    \]
    where $\n$ denotes the outward unit normal to $x + \partial \Omega^m_i$.      But note that
    $\partial \Omega^m_i = \bigcup_{j \ne i} \Sigma_{ij}$, and that the outward
    unit normal is the constant vector-field $(e_j - e_i)/\sqrt{2}$ on $\Sigma_{ij}$.
    Consequently:
    \[
      \nabla_v \modelV_i(x) = \frac{1}{\sqrt 2} \sum_{j \ne i} \inr{v}{e_j - e_i} A^m_{ij}(x) ,
    \]
    and (\ref{eq:DV}) is established. Since each of the $A^m_{ij}(x)$ is strictly positive, $L_{A^m(x)}$ is non-singular (in fact, strictly positive-definite) as a quadratic form on $E$, and hence $D\modelV(x)$ is non-singular as well. It follows that $\modelV$ is indeed a diffeomorphism, concluding the proof. 
\end{proof}

We can now give the following:
\begin{definition}
The model isoperimetric double-bubble profile $I_m : \interior \simplex \to \R$ is defined as:
\[
I_m(v) := \per(x + \Omega^m) = A^m_{12}(x) + A^m_{23}(x) + A^m_{31}(x)  \;\;\; \text{such that} \;\;\; \gamma(x + \Omega^m) = v . 
\]
\end{definition}

Lemma \ref{lem:V-diffeo} verifies that the above is well-defined. 
Thanks to the next lemma, we may (and do) extend $I_m$ by continuity to the entire $\simplex$. 

\begin{lemma}\label{lem:tripod-profile-continuous}
  $I_m$ is $C^\infty$ on $\interior \simplex$, and continuous up to $\partial
  \simplex$.  Moreover, if $v^{h} \in \interior \simplex$ converge
  to $v \in \partial \simplex$ then $I_m(v^{h}) \to \isop(\max_i v_i) =: I_m(v)$.
\end{lemma}
\begin{proof}
Since $I_m(v) = \per(\modelV^{-1}(v) + \Omega^m)$ and both $\modelV^{-1}$ and the map $x \mapsto \per(x +
  \Omega^m)$ are $C^\infty$ on their respective domains, it follows that $I_m$ is $C^\infty$ on $\interior \simplex$. 

  Now suppose that $\interior \simplex \ni v^{h} \to v \in \partial \simplex$ and let $x^{h} =
  \modelV^{-1}(v^{h})$. We first consider the case that $v$ is a corner of the
  simplex, in which case $v = (1, 0, 0)$ without loss of generality.
  Since $\gamma(x^{h} + \Omega_1^m) \to 1$, we see from~\eqref{eq:model-upper}
  that $x_1^{h} - x_2^{h} \to -\infty$ and $x_1^{h} - x_3^{h} \to -\infty$.
  In particular, $\inr{x^{h}}{n_{12}}$, $\inr{x^{h}}{n_{13}} \to \infty$, and since $t_{23} = \frac{1}{\sqrt{3}} (n_{12} + n_{13})$, 
  it follows that $\inr{x^{h}}{t_{23}} \to \infty$.
  It follows from~\eqref{eq:tripod-areas} that $\gamma^1(x^{h} + \Sigma_{ij}) \to 0$
  for every $i \ne j$, and so $I_m(v^{h}) \to 0 = \isop(1)$.

  Finally, consider the case that $v \in \partial \simplex$ but is not a corner. Without loss of generality, $v = (a, 1-a, 0)$
  for some $a \in (0, 1)$. Then $\gamma(x^{h} + \Omega_1^m) \to a$ and
  $\gamma(x^{h} + \Omega_2^m) \to 1 - a$; it follows from~\eqref{eq:model-upper}
  that
  \[
    \limsup_{n \to \infty} \frac{x_1^{h} - x_2^{h}}{\sqrt 2} \le \Phi^{-1}(1 - a) ~,~ 
    \limsup_{n \to \infty} \frac{x_2^{h} - x_1^{h}}{\sqrt 2} \le \Phi^{-1}(a).
  \]
  Since $\Phi^{-1}(1 - a) = -\Phi^{-1}(a)$, and recalling the notation of Lemma \ref{lem:model-sigma}, we conclude that
  \begin{equation}\label{eq:12-boundary}
  \inr{x^{h}}{n_{12}} =  \frac{x_2^{h} - x_1^{h}}{\sqrt 2} \to \Phi^{-1}(a).
  \end{equation}
 Now by (\ref{eq:model-lower}):
  \[
    \gamma(x^h + \Omega_m^3) \ge (1 - \Phi(\inr{x^h}{n_{23}})) (1 - \Phi(\inr{x^h}{n_{13}})) ,
  \]
  and since $\gamma(x^{(n)} + \Omega_m^3) \to 0$, at least one of
  $\inr{x^{h}}{n_{23}}$ or $\inr{x^{h}}{n_{13}}$ must diverge to $\infty$.
  But since $\inr{x^{h}}{n_{12}}$ converges to a finite limit by~\eqref{eq:12-boundary}, 
  and since $n_{12} + n_{23} + n_{31} = 0$, it follows that both
  $\inr{x^{h}}{n_{23}}$ and $\inr{x^{h}}{n_{13}}$ must diverge to $\infty$.
  By~\eqref{eq:tripod-areas}, we conclude that $\gamma^1(x^{h} + \Sigma_{23}),\gamma^1(x^{h} + \Sigma_{13}) \to 0$. 
  In addition, as $t_{12} = \frac{1}{\sqrt{3}} (n_{31} + n_{32})$, it follows that $\inr{x^h}{t_{12}} \to -\infty$. Together with~\eqref{eq:12-boundary}, the representation~\eqref{eq:tripod-areas} implies that $\gamma^1(x^{h} + \Sigma_{12}) \rightarrow \varphi(\Phi^{-1}(a)) = \isop(a)$. We thus confirm that $I_m(v^{h}) \to \isop(a) = \isop(1-a) = \isop(\max_j v_j)$, as asserted. 
\end{proof}

We have constructed the model double-bubble clusters in $E \simeq \R^2$ for $v \in \interior \simplex$ (tripod clusters), and on $\Real$ for $v \in \partial \simplex$ (halfline clusters). Clearly, by taking Cartesian products with $\Real^{n-2}$ and $\Real^{n-1}$, respectively, and employing the product structure of the Gaussian measure, these clusters extend to $\Real^n$ for all $n \geq 2$. Consequently, $I(v) \le I_m(v)$ for all $v \in \simplex$, and our goal in this work is to establish the converse inequality. 

\medskip
To this end, we observe that $I_m$ satisfies a remarkable differential equation: 
\begin{proposition}\label{prop:I_m-equation}
    At any point in $\interior \simplex$:
    \[
        \nabla I_m = \frac{1}{\sqrt 2} \modelV^{-1} \text{ and } \nabla^2 I_m = - (L_{A^m \circ \modelV^{-1}})^{-1}
    \]
    as tensors on $E$.
\end{proposition}

\begin{proof}

    Let $n_{ij} = (e_j - e_i)/\sqrt 2$ and $t_{ij} = (e_i + e_j - 2e_k)/\sqrt 6$ as in Lemma~\ref{lem:model-sigma}.
    Differentiating~\eqref{eq:tripod-areas},
    and recalling that the one-dimensional Gaussian density $\varphi$ satisfies $\varphi'(x) = -x \varphi(x)$, we calculate:
    \begin{align*}
        \nabla A^m_{ij}(x)
        &= - n_{ij} \inr{x}{n_{ij}} A^m_{ij}(x) - t_{ij} \varphi(\inr{x}{n_{ij}}) \varphi(\inr{x}{t_{ij}}) \\
        &= - n_{ij} \inr{x}{n_{ij}} A^m_{ij}(x) - t_{ij} \varphi_2(x),
    \end{align*}
    where $\varphi_2$ denotes the standard Gaussian density on $E$. Since $\sum_{(i, j) \in \cyclic} t_{ij} = 0$, 
    \[
        \nabla \sum_{(i, j) \in \cyclic} A^m_{ij}(x) = - \sum_{(i, j) \in \cyclic} A^m_{ij}(x) n_{ij} n_{ij}^T  x
        = -\frac 12 L_{A^m(x)} x.
    \]
            Since $I_m(v) = \sum_{(i, j) \in \cyclic} A^m_{ij}(\modelV^{-1}(v))$, the chain rule yields
    \[
        \nabla I_m(v) = -\frac 12  L_{A^m(\modelV^{-1}(v))} (D \modelV)^{-1}(\modelV^{-1}(v)) \modelV^{-1}(v).
    \]
    But according to~\eqref{eq:DV}, $(D\modelV)^{-1} = -\sqrt 2 L_{A^m}^{-1}$, and the first claim follows.
    The second claim follows by differentiating the first one, writing $D (\modelV^{-1}) = (D \modelV)^{-1} \circ V^{-1}$ and applying once again $(D\modelV)^{-1} = -\sqrt 2 L_{A^m}^{-1}$. 
\end{proof}

\section{Definitions and Technical Preliminaries} \label{sec:prelim}

We will be working in Euclidean space $(\Real^n,\abs{\cdot})$ endowed with a measure $\gamma = \gamma^n$ having $C^\infty$-smooth and strictly-positive density $e^{-\pot}$ with respect to Lebesgue measure.
We develop here the preliminaries we require for clusters $\Omega = (\Omega_1,\ldots,\Omega_k)$ with $k$ cells, for general $k \geq 3$, as this poses no greater generality over the case $k=3$. Recall that the cells $\Omega_i$ are assumed to be Borel, pairwise disjoint, and satisfy $\gamma(\R^n \setminus \cup_{i=1}^k \Omega_i) = 0$. In addition, they are assumed to have finite $\gamma$-weighted perimeter $P_\gamma(\Omega_i) < \infty$, to be defined below. 
We denote by $\Trips$ the collection of all cyclically ordered triplets:
\[
 \Trips := \{ \{ (a,b) , (b,c) , (c,a) \} : 1 \leq a < b < c \leq k\} .
 \]

\medskip

We write $\div X$ to denote divergence of a smooth vector-field $X$, and $\div_\gamma X$ to denote its weighted divergence:
\[
\div_{\gamma} X := \div X - \nabla_X \pot . 
\]
Given a unit vector $\n$, we write $\div_{\n^{\perp}} X$ to denote $\div X - \inr{\n}{\nabla_\n X}$, and set its weighted counterpart to be:
\[
\div_{\n^{\perp},\gamma} X = \div_{\n^{\perp}} X - \nabla_X \pot.
\]
For a smooth hypersurface $\Sigma \subset \Real^n$ co-oriented by a unit normal vector-field
$\n$, let $H_\Sigma: \Sigma \to \R$ denote its mean-curvature, i.e. the sum of its principal curvatures (equivalently, the trace of its second fundamental form). Its weighted mean-curvature $H_{\Sigma,\gamma}$ is defined as:
\[
H_{\Sigma,\gamma} := H_{\Sigma} - \nabla_\n \pot .
\]
We write $\div_\Sigma X$ for the surface divergence of a vector-field $X$ defined on $\Sigma$, i.e. $\sum_{i=1}^{n-1} \scalar{\tau_i,\nabla_{\tau_i} X}$ where $\tau_i$ is a local orthonormal frame on $\Sigma$; this coincides with $\div_{\n^{\perp}} X$ for any smooth extension of $X$ to a neighborhood of $\Sigma$. 
The weighted surface divergence $\div_{\Sigma,\gamma}$ is defined as:
\[
\div_{\Sigma,\gamma} X = \div_{\Sigma} X - \nabla_X \pot.
\]
Note that $\div_{\Sigma} \n = H_{\Sigma}$ and $\div_{\Sigma,\gamma} \n = H_{\Sigma,\gamma}$.  We will also abbreviate $\inr{X}{\n}$ by $X^\n$, and we will write $X^\tang$ for the tangential part of $X$, i.e. $X - X^{\n} \n$.  

\medskip

Given a Borel set $U \subset \Real^n$ with locally-finite perimeter, its reduced boundary $\partial^* U$ is defined (see e.g. \cite[Chapter 15]{MaggiBook}) as the subset of $\partial U$ for which there is a uniquely defined outer unit normal vector to $U$ in a measure theoretic sense. While the precise definition will not play a role in this work, we provide it for completeness. The set $U$ is said to have locally-finite (unweighted) perimeter, if for any compact subset $K \subset \Real^n$ we have:
\[
\sup \left\{\int_{U} \div X \; dx : X \in C_c^\infty(\R^n; T \R^n) \; ,\; \text{supp}(X) \subset K \; , \; |X| \le 1 \right\} < \infty . 
\]
With any Borel set with locally-finite perimeter one may associate a vector-valued Radon measure $\mu_U$ on $\Real^n$ so that:
\[
\int_U \div X \; dx = \int_{\Real^n} \inr{X}{d\mu_U} \;\;\; \forall X \in C_c^\infty(\R^n; T \R^n) .
\]
The reduced boundary $\partial^* U$ of a set $U$ with locally-finite perimeter is defined as the collection of $x \in \text{supp } \mu_U$ so that the vector limit:
\[
\n_U := \lim_{\eps \to 0+} \frac{\mu_U(B(x,\eps))}{\abs{\mu_U}(B(x,\eps))}  
\]
exists and has length $1$ (here $\abs{\mu_U}$ denotes the total-variation of $\mu_U$ and $B(x,\eps)$ is the open Euclidean ball of radius $\eps$ centered at $x$). When the context is clear, we will abbreviate $\n_U$ by $\n$. It is known that $\partial^* U$ is a Borel subset of $\partial U$ and that $\abs{\mu_U}(\R^n \setminus \partial^* U) = 0$. The relative perimeter of $U$ in a Borel set $F$, denoted $P(U ; F)$, is defined as $\abs{\mu_U}(F)$. 

Recall that the $\gamma$-weighted perimeter of $U$ was defined in the Introduction as:
\[
\per(U) := \sup \left\{\int_U \div_\gamma X \, d\gamma: X \in C_c^\infty(\R^n; T \R^n), |X| \le 1 \right\}.
\]
Clearly, if $U$ has finite weighted-perimeter $\per(U) < \infty$, it has locally-finite (unweighted) perimeter. It is known \cite[Theorem 15.9]{MaggiBook} that in that case:
\[
\per(U) = \gamma^{n-1}(\partial^* U)  ,
\]
where as usual:
\[
\gamma^{n-1} := e^{-W} \H^{n-1} . 
\]
In addition, by the Gauss--Green--De Giorgi theorem, the following integration by parts formula holds for any $C_c^1$ vector-field $X$ on $\R^n$, and Borel subset $U \subset \Real^n$ with $\per(U) < \infty$ (the proof in \cite[Theorem 15.9]{MaggiBook} immediately carries over to the weighted setting):
\begin{equation}\label{eq:integration-by-parts}
  \int_U \div_\gamma X \, d\gamma^n = \int_{\partial^* U} X^\n \, d\gamma^{n-1} .
\end{equation}

\subsection{Volume and Perimeter Regular Sets}

\begin{definition}[Admissible Vector-Fields]
A vector-field $X$ on $\R^n$ is called \emph{admissible} if it is $C^\infty$-smooth and satisfies:
\begin{equation} \label{eq:field-bdd}
\forall i \geq 0 \;\;\; \max_{x \in \R^n} \norm{\nabla^i X(x)} \leq C_i < \infty .
\end{equation}
\end{definition}
\noindent
Any smooth vector-field which is compactly-supported is clearly admissible, but we will need to use more general vector-fields in this work. 

Let $F_t$ denote the associated flow along an admissible vector-field $X$, defined as the family of maps $\set{F_t : \R^n \to \R^n}$ solving
the following ODE:
\[
\frac{d}{dt} F_t(x) = X \circ F_t(x) ~,~ F_0(x) = x .
\]
It is well-known that a unique smooth solution in $t \in \Real$ exists for all $x \in \R^n$, and that the resulting maps $F_t : \R^n \rightarrow \R^n$ are $C^\infty$ diffeomorphisms, so that the partial derivatives in $t$ and $x$ of any fixed order are uniformly bounded in $(x,t) \in \Real^n \times [-T,T]$, for any fixed $T > 0$.

Note that if $\Omega=(\Omega_1,\ldots,\Omega_k)$ is a cluster then so is $F_t(\Omega) = (F_t(\Omega_1),\ldots,F_t(\Omega_k))$, since its cells remain Borel, pairwise disjoint, as well as of $\gamma$-weighted finite-perimeter and satisfying $\gamma(\Real^n \setminus \cup_i F_t(\Omega_i)) = \gamma(F_t(\Real^n \setminus \cup_i \Omega_i)) = 0$ since $F_t$ is a Lipschitz map.  
We define the $r$-th variations of weighted volume and perimeter of $\Omega$ as:
\begin{align*}
  \delta_X^r V(\Omega) &:= \left. \brac{\frac{d}{dt}}^r\right|_{t=0} \gamma(F_t(\Omega)) ,\\
  \delta_X^r A(\Omega) &:= \left. \brac{\frac{d}{dt}}^r \right|_{t=0} \per(F_t(\Omega)) ,
\end{align*}
whenever the right-hand sides exist. When $\Omega$ is clear from the context, we will simply write $\delta_X^r V$ and $\delta_X^r A$; when $r = 1$, we will write $\delta_X V$ and $\delta_X A$. 

\smallskip

It will be of crucial importance for us in this work to calculate the first and second variations of weighted volume and perimeter for \textbf{non-}compactly supported (albeit simple) vector-fields, for which even the existence of $\delta_X^r V(\Omega)$ and especially $\delta_X^r A(\Omega)$ is not immediately clear.
Indeed, even for the case of the standard Gaussian measure, the derivatives of its density are asymptotically larger at infinity than the Gaussian density itself.
We consequently 
introduce the following:

\begin{definition}[Volume / Perimeter Regular Set]
    A Borel set $U$ is said to be \emph{volume regular} with respect to the measure $\gamma$ (which, recall, has density $e^{-W}$ with respect to Lebesgue) if:
\[
\forall i,j \geq 0 \;\;\; \exists \delta > 0 \;\; \int_{U} \sup_{z \in B(x,\delta)} \norm{\nabla^i \pot(z)}^j e^{-\pot(z)} dx < \infty ~,~ 
\]
It is said to be \emph{perimeter regular} with respect to the measure $\gamma$ if:
\[
\forall i,j \geq 0 \;\;\; \exists \delta > 0 \;\; \int_{\partial^* U} \sup_{z \in B(x,\delta)} \norm{\nabla^i \pot(z)}^j e^{-\pot(z)} d\H^{n-1}(x) < \infty .
\]
If $\delta > 0$ above may be chosen uniformly for all $i,j \geq 0$, $U$ is called \emph{uniformly} volume / perimeter regular. 
\end{definition}
\noindent
Here $\norm{\cdot}$ denotes the Hilbert-Schmidt norm of a tensor, defined as the square-root of the sum of squares of its coordinates in any local orthonormal frame. 
Note that volume (perimeter) regular sets clearly have finite weighted volume (perimeter).

\medskip
Write $JF_t = \text{det}(dF_t)$ for the Jacobian of $F_t$, and observe that by the change-of-variables formula for smooth injective functions:
\begin{equation} \label{eq:Jac-vol}
\gamma(F_t(U)) = \int_{U} J F_t  e^{-\pot \circ F_t} \, dx,
\end{equation}
for any Borel set $U$. Similarly, if $U$ is in addition of locally finite-perimeter, let $\Phi_t = F_t|_{\partial^* U}$ and write $J \Phi_t = \text{det}((d_{n_U^{\perp}} F_t)^T d_{n_U^{\perp}} F_t)^{1/2}$ for the Jacobian of $\Phi_t$ on $\partial^* U$. Since $\partial^* U$ is locally $\H^{n-1}$-rectifiable, \cite[Theorem 11.6]{MaggiBook} implies:
\begin{equation} \label{eq:Jac-area}
\gamma^{n-1}(F_t(\partial^* U)) = \int_{\partial^* U} J \Phi_t e^{-\pot \circ F_t} \, d\H^{n-1} . 
\end{equation}

\begin{lemma} \label{lem:regular}
\hfill
\begin{enumerate}
\item
If $U$ is volume regular with respect to $\gamma$ then for any $r \geq 1$, $t \mapsto \gamma(F_t(U))$ is $C^r$ in an open neighborhood of $t=0$, and:
\[
\delta_X^r V(U) = \int_{U} \frac{d^r}{(dt)^r} (J F_t e^{-\pot \circ F_t}) dx .
\]
Furthermore, if $U$ is uniformly volume regular then there exists an open neighborhood of $t=0$ where $t \mapsto \gamma(F_t(U))$ is $C^\infty$. 
\item
If $U$ is perimeter regular with respect to $\gamma$ then for any $r \geq 1$, $t \mapsto \per(F_t(U))$ is $C^r$ in an open neighborhood of $t=0$, and:
\[
\delta_X^r A(U) = \int_{\partial^* U} \frac{d^r}{(dt)^r} (J \Phi_t e^{-\pot \circ F_t}) dx .
\]
Furthermore, if $U$ is uniformly perimeter regular then there exists an open neighborhood of $t=0$ where $t \mapsto \per(F_t(U))$ is $C^\infty$.
\end{enumerate}
\end{lemma}

For the proof of the second part, we require a simple:
\begin{lemma} \label{lem:F-reduced}
For any Borel set $U \subset \R^n$ with $\per(U) < \infty$ and diffeomorphism $F : \R^n \rightarrow \R^n$:
\[
\gamma^{n-1}(\partial^* F(U)) = \gamma^{n-1}(F(\partial^* U)) . 
\]
\end{lemma}
\begin{proof}
    This follows from~\cite[Proposition 17.1]{MaggiBook}, along with the fact
    that $\gamma^{n-1}$ is absolutely continuous with respect to $\calH^{n-1}$.
\end{proof}

\begin{proof}[Proof of Lemma \ref{lem:regular}]
In view of Lemma \ref{lem:F-reduced}, our task is to justify taking derivative inside the integral representations (\ref{eq:Jac-vol}) and (\ref{eq:Jac-area}). Taking difference quotients, applying Taylor's theorem with Lagrange remainder and induction on $r$, it is enough to establish by Dominant Convergence Theorem that for some $\eps > 0$:
\begin{align} 
\label{eq:dominant1} & \int_U \sup_{t \in [-\eps,\eps]} \frac{d^r}{(dt)^r} (J F_t(x) e^{-\pot(F_t(x))}) dx < \infty ~,\\
\label{eq:dominant2} & \int_{\partial^* U} \sup_{t \in [-\eps,\eps]} \frac{d^r}{(dt)^r} (J \Phi_t(x) e^{-\pot(F_t(x))}) d\H^{n-1}(x) < \infty .
\end{align}
By the Leibniz product rule, for $\F = F,\Phi$:
\[
\frac{d^r}{(dt)^r} (J \F_t(x) e^{-\pot(F_t(x))}) = \sum_{p+q=r} {r \choose p} \frac{d^p}{(dt)^p} J \F_t(x) \frac{d^{q}}{(dt)^{q}}  e^{-\pot(F_t(x))} .
\]
For each $x$, $t \mapsto J \F_t(x)$ is a smooth function of the differential $d F_t(x)$ (note that for $\F = \Phi$, the normal $n_{U}(x)$ remains fixed for all $t$).
This differential satisfies:
\[
\frac{d}{dt} dF_t(x) = \nabla X(F_t(x))  dF_t(x) , 
\]
and since $X$ satisfies (\ref{eq:field-bdd}), it follows that $\sup_{t \in [-\eps,\eps]} \frac{d^p}{(dt)^p} J \F_t(x)$ is uniformly bounded in $x \in \Real^n$ for all fixed $\eps > 0$ and $p$. 

It remains to handle the $\frac{d^{q}}{(dt)^{q}}  e^{-\pot(F_t(x))}$ term. Repeated differentiation and application of the chain rule results in a polynomial expression in $\nabla^a \pot$ and $\nabla^b X$ time $e^{-\pot}$ evaluated at $F_t(x)$, where the degree of the polynomial, $a$ and $b$ are bounded by a function of $q$. Since $X$ satisfies (\ref{eq:field-bdd}), we may bound the magnitudes of $\nabla^b X$ by a constant depending on $q$. Now let $\delta > 0$ be given by the assumption that $U$ is volume and perimeter regular with respect to $\gamma$, for appropriately bounded above $i,j$. Since $\abs{F_t(x) - x} \leq \abs{t} \max_{x \in \R^n} \abs{X(x)}$, we may find $\eps > 0$ so that $\abs{F_t(x) - x} \leq \delta$ uniformly in $x \in \R^n$. It follows that for an appropriate constant $D_r >0$:
\[
\sup_{t \in [-\eps,\eps]} \frac{d^r}{(dt)^r} (J \F_t(x) e^{-\pot(F_t(x))}) \leq D_r \sup_{z \in B(x,\delta)} \norm{\nabla^i \pot(z)}^j e^{-\pot(z)} ,
\]
and so the required (\ref{eq:dominant1}) and (\ref{eq:dominant2}) are established by definition of a regular set. 

Finally, when the set is assumed to be uniformly regular, $\delta > 0$ and hence $\eps  > 0$ above may be chosen uniformly in $r \geq 1$, and $C^\infty$ smoothness is established for $t \in (-\eps,\eps)$. 
\end{proof}

\subsection{Cutoff Function}

The following lemma will be very useful in this work for calculating first and second variations of non-compactly-supported vector-fields. In several of these instances the vector-field will not be bounded in magnitude, and so we formulate it quite generally. 

\begin{definition}[Cutoff function $\eta_R$]
Given $R > 0$, we denote by $\eta_R : \R^n \rightarrow [0,1]$ a smooth compactly-supported cutoff function on $\R^n$ with $\eta_R(x) = 1$ for $\abs{x} \leq R$ and $\abs{\nabla \eta_R} \leq 1$.
\end{definition}

\noindent
We also denote by $B_R$ the open Euclidean ball of radius $R>0$ centered at the origin. 

\begin{lemma} \label{lem:cutoff}
\hfill
\begin{enumerate}
\item
Let $X$ denote a $C^1$ vector-field on $\R^n$ so that $\abs{X}, \abs{\nabla X} \leq P(\abs{\nabla \pot},\ldots,\norm{\nabla^p \pot})$ for some real valued polynomial $P$ and $p \geq 1$. Assume that $U$ is volume regular with respect to the measure $\gamma$. Then: 
\[
\int_U \div_\gamma X \; d\gamma = \lim_{R \rightarrow \infty} \int_U \div_{\gamma}( \eta_R X) \; d\gamma . 
\]
\item
Let $X$ denote a $C^1$ vector-field on a smooth hypersurface $\Sigma \subset \R^n$ so that $\abs{X}, \abs{\nabla_{\Sigma} X} \leq P(\abs{\nabla \pot})$ for some real valued polynomial $P$ and $p \geq 1$. Assume that $\Sigma \subset \partial^*U$ with $U$ being perimeter regular with respect to the measure $\gamma$. Then:
\[
\int_{\Sigma} \div_{\Sigma,\gamma} X \; d\gamma^{n-1} = \lim_{R \rightarrow \infty} \int_\Sigma \div_{\Sigma,\gamma}( \eta_R X) \; d\gamma^{n-1} . 
\]
\end{enumerate}
\end{lemma}
\begin{proof}
We will prove the first part, the proof of the second one is identical. Write: 
\[
\int_U \div_{\gamma} X d\gamma = \int_U \div_\gamma (\eta_R X) d\gamma - \int_U \scalar{\nabla \eta_R,X} d\gamma + \int_U (1-\eta_R) \div_{\gamma} X d\gamma . 
\]
Next, note that:
\[
\abs{\div_\gamma X} \leq \abs{\div X} + \abs{\nabla_X \pot} \leq Q(\abs{\nabla \pot},\ldots,\norm{\nabla^p \pot}) ~,~ \abs{\scalar{\nabla \eta_R,X}} \leq P(\abs{\nabla \pot},\ldots,\norm{\nabla^p \pot}) ,
\]
for an appropriate polynomial $Q$. Consequently:
\begin{align*}
\abs{\int_U (1-\eta_R) \div_{\gamma} X d\gamma} & \leq \int_{U \setminus B_R} Q(\abs{\nabla \pot},\ldots,\norm{\nabla^p \pot})  d\gamma ~, \\
\abs{\int_U \scalar{\nabla \eta_R,X}  d\gamma} & \leq \int_{U \setminus B_R} P(\abs{\nabla \pot},\ldots,\norm{\nabla^p \pot})  d\gamma ~,
\end{align*}
and volume regularity implies that both of these terms go to zero as $R \rightarrow \infty$, yielding the claim. 
\end{proof}

\subsection{First Variation of Weighted Volume and Perimeter}

We now derive formulas for the first variations of (weighted) volume and perimeter. While these are well-known for sets with smooth boundaries and compactly supported vector-fields (see e.g. \cite{RCBMIsopInqsForLogConvexDensities}), we put special emphasis on not assuming anything on the set nor vector-field beyond what we need for the proof, as this will be important for us in the sequel. 

\begin{proposition}\label{prop:first-variation}
Let $X$ be an admissible vector-field on $\R^n$, and let $U \subset \R^n$ denote a Borel subset. \\
\begin{enumerate}
\item
If $U$ is volume regular with respect to the measure $\gamma$ and satisfies $\per(U) < \infty$, then:
  \begin{equation}\label{eq:formula-first-variation-of-volume}
    \delta_X V(U) = \int_{\partial^* U} X^\n \, d\gamma^{n-1} .
  \end{equation}
\item
If $U$ is perimeter regular with respect to the measure $\gamma$ then:
\[
\delta_X A(U) = \int_{\partial^* U} \div_{\n_U^{\perp},\gamma} X d\gamma^{n-1} .
\]
If in addition  $\partial^* U = \Sigma \cupdot \Xi$ where $\Sigma$ is a smooth hypersurface and $\H^{n-1}(\Xi) = 0$, then: 
  \begin{equation}\label{eq:formula-first-variation-of-area-before}
    \delta_X A(U) = \int_{\Sigma} (H_{\Sigma,\gamma} X^\n + \div_{\Sigma,\gamma} X^\tang) \, d\gamma^{n-1}.
  \end{equation}
 \end{enumerate}
\end{proposition}

\begin{proof}Let $F_t$ be the flow along $X$. Under the assumptions of the first assertion, Lemma \ref{lem:regular} implies:
\[
\delta_X V(U) = \int_U \left . \frac{d}{dt} \right |_{t=0} (J F_t e^{-\pot \circ F_t}) dx . 
\]
It is well-known (e.g. \cite[(2.13)]{SternbergZumbrun}) that $\frac{d}{dt} |_{t=0} J F_t = \div X$, and therefore:
\[
  \delta_X V(U) = \int_U \brac{\div X - \nabla_X \pot} d\gamma = \int_{U} \div_\gamma X \, d\gamma .
\]  
In order to apply integration-by-parts, we need to first make $X$ compactly supported. Applying Lemma \ref{lem:cutoff} and integrating by parts the compactly-supported vector-field $\eta_R X$ using ~\eqref{eq:integration-by-parts}, we deduce:
\[
\int_U \div_{\gamma} X d\gamma = \lim_{R \rightarrow \infty} \int_U \div_\gamma (\eta_R X) d\gamma = \lim_{R \rightarrow \infty} \int_{\partial^* U} \eta_R X^{\n} d\gamma^{n-1} = \int_{\partial^* U} X^{\n} d\gamma^{n-1} ,
\]
where the last equality follows by Dominant Convergence since $X$ is uniformly bounded and $\per(U) < \infty$.

For the second assertion, set as usual $\Phi_t = F_t|_{\partial^* U}$ and recall that $J \Phi_t = \text{det}((d_{\n_{U}^{\perp}} F_t)^T d_{\n_{U}^{\perp}} F_t)^{1/2}$ is the Jacobian of $\Phi_t$. Under the assumptions of the second assertion, Lemma \ref{lem:regular} implies:
\[
\delta_X A(U) = \int_{\partial^* U} \left . \frac{d}{dt} \right |_{t=0} (J \Phi_t e^{-\pot \circ F_t}) d\H^{n-1}(x) . 
\]
It is well-known (e.g. \cite[(2.16)]{SternbergZumbrun}) that $\frac{d}{dt} |_{t=0} J \Phi_t = \div_{\n_U^\perp} X$, and hence: 
\[
\delta_X A(U) = \int_{\partial^* U} \brac{\div_{\n_U^\perp} X - \nabla_X \pot} = \int_{\partial^* U} \div_{\n_U^\perp,\gamma} X \, d\gamma^{n-1} .
\]
To establish the last claim, simply continue as follows:
\begin{align*}
& = \int_{\Sigma} \div_{\Sigma,\gamma} X \, d\gamma^{n-1} \notag \\
& = \int_{\Sigma} \brac{\div_{\Sigma,\gamma} X^\tang + \div_{\Sigma,\gamma} (X^\n \n) } d\gamma^{n-1}  \notag \\
&= \int_{\Sigma} \brac{\div_{\Sigma,\gamma} X^\tang + H_{\Sigma,\gamma} X^\n } d\gamma^{n-1}.
\end{align*}
\end{proof}

\section{Isoperimetric Minimizing Clusters} \label{sec:first-order}

Given a cluster $\Omega = (\Omega_1,\ldots,\Omega_k)$, we define the interface between the $i$-th and $j$-th cells ($i \neq j$) as:
\[
\Sigma_{ij} = \Sigma_{ij}(\Omega) := \partial^* \Omega_i \cap \partial^* \Omega_j ,
\]
and set:
\[
A_{ij} = A_{ij}(\Omega) := \gamma^{n-1}(\Sigma_{ij}). 
\]
It is standard to show (see \cite[Exercise 29.7, (29.8)]{MaggiBook}) that for any $S \subset \set{1,\ldots,k}$:
\begin{equation} \label{eq:nothing-lost}
\H^{n-1}\brac{\partial^*(\cup_{i \in S} \Omega_i) \setminus \cup_{i \in S , j \notin S} \Sigma_{ij}} = 0 .
\end{equation}
In particular:
\[
\per(\Omega_i) = \sum_{j \neq i} A_{ij}(\Omega) , 
\]
and hence:
\[
\per(\Omega) = \frac{1}{2} \sum_{i=1}^k \per(\Omega_i) = \sum_{i < j} A_{ij}(\Omega)  .
\]

\medskip

A cluster $\Omega$ is called an \emph{isoperimetric minimizer} if $\per(\Omega') \ge
\per(\Omega)$ for every other cluster $\Omega'$ satisfying $\gamma(\Omega') = \gamma(\Omega)$.
The following theorem, due to Almgren \cite{AlmgrenMemoirs} (see also the exposition by Morgan \cite[Chapter 13]{MorganBook5Ed} and simplified presentation by Maggi \cite[Chapters 29-30]{MaggiBook}), summarizes the results we will need from Geometric Measure Theory on the existence and regularity of minimizing clusters:

\begin{theorem}[Almgren] \label{thm:Almgren}
Let $(\Real^n,\abs{\cdot},\gamma = e^{-W} dx)$ with $W \in C^\infty(\Real^n)$ so that $\gamma$ is a probability measure. Let $\simplex = \{ v \in \Real^k : v_i \geq 0 , \sum_{i=1}^k v_i = 1 \}$. 
\begin{enumerate}[(i)]
\item 
For any prescribed $v \in \simplex$, an isoperimetric minimizing $k$-cluster $\Omega$ satisfying $\gamma(\Omega) = v$ exists. 
\item Moreover, $\Omega$ may be chosen so that all of its cells are open, and for every $i$, $\gamma^{n-1}(\partial \Omega_i \setminus \partial^* \Omega_i) = 0$. In particular, for all $i \neq j$, $\gamma^{n-1}(\partial \Omega_i  \cap \partial \Omega_j \setminus \Sigma_{ij}) = 0$. 
\item For all $i \neq j$ the interfaces $\Sigma_{ij}$ are $C^\infty$ smooth $(n-1)$-dimensional submanifolds, relatively open in $\partial \Omega_i \cap \partial \Omega_j$, and for every $x \in \Sigma_{ij}$ there exists $\epsilon > 0$ such that $B(x,\epsilon) \cap \Omega_l = \emptyset$ for all $l \neq i,j$. 
\end{enumerate}
\end{theorem}
\begin{proof}
\hfill
\begin{enumerate}[(i)]
\item
  It is well-known (e.g.~\cite[Proposition~12.15]{MaggiBook}) that the (weighted)
  perimeter is lower semi-continuous with respect to (weighted) $L^1$
  convergence: if $U^r \to U$ in $L^1(\gamma)$ then $\liminf_r \per(U^r) \ge
  \per(U)$ for all Borel sets $U^r$ of finite (weighted) perimeter. Clearly, the
  same applies to clusters, where $L^1(\gamma)$ convergence is understood for each of
  the individual cells.
  
   It is also well-known that, since $\gamma$ has finite mass, the set $\set{ U \in \B(\R^n) : \per(U) \leq C}$ (where
  $\B(\R^n)$ denotes the collection of Borel subsets of $\R^n$) is compact in $L^1(\gamma)$ -- for bounded sets, this follows from ~\cite[Theorem~12.26]{MaggiBook}, and the general case follows by truncation with a large ball and a standard diagonalization argument (see e.g. \cite[Theorem 2.1]{RitoreRosalesMinimizersInEulideanCones}).
  
  Set $I(v) := \inf \set{ \per(\Omega) : \text{$\Omega$ is a $k$-cluster with }\gamma(\Omega) = v}$. As the latter set is clearly non-empty, obviously $I(v) < \infty$. Given $v \in \Delta$, let $\Omega^r$ be a sequence of $k$-clusters with $\gamma(\Omega^r) = v$ and $\per(\Omega^r) \rightarrow I(v)$. As $\per(\Omega^r_i) \leq \per(\Omega^r) \leq I(v)+1$ for large enough $r$, by passing to a subsequence, it follows that each of the cells $\Omega^r_i$ converges in $L^1(\gamma)$ to $\Omega_i$. By Dominated Convergence (as the total mass is finite), we must have $\gamma(\Omega_i) = v_i$, and the limiting $\Omega$ is easily seen to be a cluster (possibly after a measure-zero modification to ensure disjointness of the cells). It follows by lower semi-continuity that:
  \[
    I(v) \le \per(\Omega) \le \liminf_{r \to \infty} \per(\Omega^r) = I(v),
  \]
  and consequently $\per(\Omega) = I(v)$. Hence $\Omega$ is a minimizing cluster with $\gamma(\Omega) = v$. 
  
  Note that the proof is much simpler than the one in the unweighted setting, where the total mass is infinite, but on the other hand perimeter and volume are translation-invariant. 
  
  \item That $\Omega$ may be chosen so that $\gamma^{n-1}(\partial \Omega_i \setminus \partial^* \Omega_i) = 0$ for all $i$ follows from \cite[Theorem 30.1]{MaggiBook}; the proof carries over to the weighted setting (see below). In particular, the topological boundary of each cell has zero $\gamma^n$-measure, and so by replacing each cell by its interior, we do not change its measure nor its reduced boundary (and hence its $\gamma$-weighted perimeter), so $\Omega$ remains an isoperimetric minimizer. As this operation can only reduce the topological boundary, it still holds that $\gamma^{n-1}(\partial \Omega_i \setminus \partial^* \Omega_i) = 0$ for all $i$. In particular, since:
  \[
  \partial \Omega_i  \cap \partial \Omega_j \setminus [\partial^* \Omega_i \cap \partial^* \Omega_j] \subset [\partial \Omega_i  \setminus \partial^* \Omega_i] \cup
  [\partial \Omega_j  \setminus \partial^* \Omega_j] ,
  \]
  the final assertion follows. 
  \item
 The assertions follow from \cite[Theorem 30.1, Lemma 30.2 and Corollary 3.3]{MaggiBook}, whose proofs carry over to the weighted setting. Indeed, all of the arguments used in those proofs are local in nature, and so as long as our density $e^{-W}$ is $C^{\infty}$-smooth and positive, and hence locally bounded above and below (away from zero), the proofs of $\gamma^{n-1}(\partial \Omega_i \setminus \partial^* \Omega_i) = 0$, of the relative openness of $\Sigma_{ij}$ in $\partial \Omega_i \cap \partial \Omega_j$, and of the disjointness of $B(x,\eps)$ from $\Omega_\ell$ carry over by adjusting constants (which are local). By \cite[Corollary 3.3 and Remark 3.4]{MaggiBook}, we also know that for any $x \in \Sigma_{ij}$ there exists $r_x > 0$ so that $\Omega_i$ and $\Omega_j$ are measure-constrained weighted-perimeter minimizers in $B(x,r_x)$. Consequently, by regularity theory for volume-constrained perimeter minimizers (see Morgan \cite[Section 3.10]{MorganRegularityOfMinimizers} for an adaptation to the weighted Riemannian setting), since our density is $C^{\infty}$-smooth, it follows that $\Sigma_{ij} \cap B(x,r_x) = \partial \Omega_i \cap \partial \Omega_j \cap B(x,r_x)$ is a $C^{\infty}$-smooth $(n-1)$-dimensional submanifold.  

\end{enumerate}
\end{proof}

Given a minimizing cluster with (smooth) interfaces $\Sigma_{ij}$, let $\n_{ij}$ be the unit normal field along $\Sigma_{ij}$ that
points from $\Omega_i$ to $\Omega_j$. We use $\n_{ij}$ to co-orient $\Sigma_{ij}$, and since $\n_{ij} = -\n_{ji}$, note that $\Sigma_{ij}$ and $\Sigma_{ji}$ have opposite orientations.
When $i$ and $j$ are clear from the context, we will simply write $\n$. 
We will typically abbreviate $H_{\Sigma_{ij}}$ and $H_{\Sigma_{ij},\gamma}$ by $H_{ij}$ and $H_{ij,\gamma}$, respectively. 

\subsection{Volume and Perimeter Regularity}

For simplicity, we assume henceforth that $\gamma$ is a probability measure and that:
\begin{equation} \label{eq:decreasing}
\exists R_* \geq 0 \;\;\; [R_*,\infty) \ni R \mapsto \gamma^{n-1}(\partial B_R)  \text{ is decreasing.}
\end{equation}

\begin{lemma} \label{lem:regular-cond}
For each $i \geq 0$, let $f_i : \Real_+ \rightarrow \Real_+$ denote a $C^1$ increasing function so that:
\[
\norm{\nabla^i W(x)} \leq f_i(\abs{x}) \;\;\; \forall x \in \R^n . 
\]
Assume that for all $i,j \geq 0$, there exists $\delta_{ij} > 0$, so that defining $F_{ij} : \Real_+ \rightarrow \Real_+$ by:
\[
F_{ij}(R) := f_i(R + \delta_{ij})^j e^{\delta_{ij} f_1(R+\delta_{ij})} ,
\]
we have:
\begin{equation} \label{eq:integrability-conds}
\int_{\R^n} F_{ij}(\abs{x}) d\gamma < \infty ~,~ \int_{\R^n} F'_{ij}(\abs{x}) d\gamma < \infty . 
\end{equation} 
Then any Borel set $U \subset \R^n$ is volume regular, and the cells of any isoperimetric minimizing cluster are perimeter regular, with respect to $\gamma$. \\
Furthermore, if $\delta_{ij} \geq \delta > 0$ uniformly in $i,j \geq 0$, the above sets are in fact uniformly volume and perimeter regular, respectively. 
\end{lemma}

For the proof of the perimeter regularity, we require the following perimeter decay estimate:
\begin{lemma}\label{lem:area-decay}
    If $\Omega = (\Omega_1,\ldots,\Omega_k)$ is an isoperimetric minimizing cluster then:
    \[
        \sum_{i=1}^k \gamma^{n-1}(\partial^* \Omega_i \setminus B_R) \leq 3 k \gamma^{n-1}(\partial B_R) 
    \]
    for all $R \geq R_*$. 
\end{lemma}
\begin{proof}
   Let $R \geq R_*$. 
   Since $\gamma(\R^n \setminus B_R) = \sum_{i=1}^k \gamma(\Omega_i \setminus B_R)$, we may choose a non decreasing sequence $R = R_0 \leq R_1 \leq R_2 \leq \ldots R_{k-1} \leq R_k = \infty$ so that $\gamma(B_{R_i} \setminus B_{R_{i-1}}) = \gamma(\Omega_i \setminus B_R)$ for all $i=1,\ldots,k$. 
   Now define the cells of a competing cluster $\tilde \Omega$ as follows:
    \[
        \tilde \Omega_i := (\Omega_i \cap B_R) \cup (B_{R_i} \setminus B_{R_{i-1}})  .
    \]
    Clearly $\gamma(\tilde \Omega) = \gamma(\Omega)$.
                    Now observe that (see e.g. \cite[Lemma 12.22, Theorem 16.3]{MaggiBook}):
    \begin{align*}
        \per(\tilde \Omega_i)
        &\le \per(\Omega_i \cap B_R) + \per(B_{R_i} \setminus B_{R_{i-1}}) \\
        & \le \gamma^{n-1}(\partial^* \Omega_i \cap B_R) + \gamma^{n-1}(\partial B_R) + \gamma^{n-1}(\partial B_{R_i}) + \gamma^{n-1}(\partial B_{R_{i-1}}) \\
        &\le \gamma^{n-1}(\partial^* \Omega_i\cap B_R) + 3 \gamma^{n-1}(\partial B_R) .
    \end{align*}
   Summing in $i$ and dividing by $2$,
    \[
        \per(\tilde \Omega) \leq \frac{1}{2} \sum_{i=1}^k \gamma^{n-1}(\partial^* \Omega_i\cap B_R) + \frac{3}{2} k \gamma^{n-1}(\partial B_R) .
            \]
    On the other hand, the fact that $\Omega$ is minimizing and $\gamma(\Omega) = \gamma(\tilde \Omega)$ implies that
    \[
        \per(\tilde \Omega) \ge \per(\Omega) = \frac{1}{2} \sum_{i=1}^k \gamma^{n-1}(\partial^* \Omega_i) .
    \]
    Combining these inequalities yields the assertion.
\end{proof}

\begin{proof}[Proof of Lemma \ref{lem:regular-cond}]
Note that by the intermediate value theorem: 
\[
\abs{\pot(z) -\pot(x)} \leq \abs{z-x} f_1(\max(\abs{x},\abs{z})) .
\]
Consequently:
\[
\sup_{z \in B(x,\delta_{ij})} \norm{\nabla^i \pot(z)}^j e^{-\pot(z)} \leq f_i(\abs{x}+\delta_{ij})^j e^{\delta_{ij} f_1(\abs{x}+\delta_{ij})} e^{-\pot(x)} = F_{ij}(\abs{x}) e^{-\pot(x)} ,
\]
and so if for any $i,j \geq 0$, the right-hand function is integrable on $\R^n$, then all Borel sets are volume regular with respect to $\gamma$. 

Now if $\Omega = (\Omega_1,\ldots,\Omega_k)$ is an isoperimetric minimizing cluster, then for each of its cells $U$:
\[
\int_{\partial^* U} \sup_{z \in B(x,\delta_{ij})} \norm{\nabla^i \pot(z)}^j e^{-\pot(z)} d\H^{n-1} \leq \int_{\partial^* U} F_{ij}(\abs{x}) d\gamma^{n-1}(x) .
\]
Integrating by parts and applying Lemma \ref{lem:area-decay}, it follows that:
\begin{align*}
& \leq F_{ij}(R_*) \gamma^{n-1}(\partial^* U) + \int_{R_*}^\infty F_{i,j}'(R) \gamma^{n-1}(\partial^* \Omega_i \setminus B_R) dR \\
& \leq  F_{ij}(R_*) \gamma^{n-1}(\partial^* U) + 3 k \int_{R_*}^\infty F_{i,j}'(R) \gamma^{n-1}(\partial B_R) dR \\
& =  F_{ij}(R_*) \gamma^{n-1}(\partial^* U) + 3 k \int_{\R^n \setminus B_{R_*}} F'_{i,j}(\abs{x}) d\gamma(x) . 
\end{align*}
Since $\gamma^{n-1}(\partial^* U) < \infty$ for any cell of a minimizing cluster, the perimeter regularity of $U$ with respect to $\gamma$ follows as soon as the second term above is finite for all $i,j \geq 0$, as asserted. 
\end{proof}

\begin{corollary} \label{cor:Gaussian-regular}
For the standard Gaussian measure $\gamma$, any Borel set $U \subset \R^n$ is uniformly volume regular, and the cells of any isoperimetric minimizing cluster are uniformly perimeter regular, with respect to $\gamma$. 
\end{corollary}
\begin{proof}
Note that (\ref{eq:decreasing}) indeed holds since $\gamma^{n-1}(\partial B_R) = c_n R^{n-1} e^{-R^2/2}$. Setting $f_0(R) = \frac{R^2}{2}$, $f_1(R) = R$, $f_2(R) = \sqrt{n}$, $f_i(R) = 0$ for all $i \geq 3$ and $\delta = 1$, it is immediate to verify the integrability conditions (\ref{eq:integrability-conds}), as the Gaussian measure decays faster than any polynomial or exponential function. The assertion therefore follows by Lemma \ref{lem:regular-cond}. 
\end{proof}

We henceforth proceed under the following assumption (note that as long we take only a finite number of variations in our analysis, the uniformity assumption may be dropped):
\begin{equation} \label{eq:gamma-regular}
\begin{array}{l}
\text{$\gamma$ is a probability measure for which all cells of any isoperimetric}\\
\text{minimizing cluster are uniformly volume and perimeter regular.}
\end{array}
\end{equation}

\subsection{Stationarity}

In this subsection, we show that a minimizing cluster satisfies the stationarity property: 
\[
\delta_X V = 0 \;\; \Rightarrow \;\; \delta_X A = 0 ,
\]
for all admissible vector-fields $X$. When $X$ is compactly supported this is well-known and standard in the single-bubble case, and was proved for Euclidean double-bubbles in \cite{DoubleBubbleInR3} assuming higher-order boundary regularity, which we avoid in this work (see Remark \ref{rem:no-higher-regularity} below). 

In what follows, let (\ref{eq:gamma-regular}) hold, let $\Omega$ be a minimizing cluster, and let $\Sigma_{ij}$ be its interfaces. Recall that $E = \set{x \in \R^k : \sum_{i=1}^k x_i = 0}$. 

\begin{lemma}\label{lem:span}
    If $\gamma(\Omega) \in \interior \simplex$ then the set $\{e_i - e_j: \gamma^{n-1}(\Sigma_{ij}) > 0\}$ spans $E$.
\end{lemma}

\begin{proof}
    Suppose in the contrapositive that $\{e_i - e_j: \gamma^{n-1}(\Sigma_{ij}) > 0\}$ does not span $E$.
    Then there exists some non-zero $v$ in $E$ such that $v_i = v_j$ whenever $\gamma^{n-1}(\Sigma_{ij}) > 0$.
    Let $S = \{i : v_i < 0\}$ and let $T = \{i : v_i \ge 0\}$. Since $v \in E$ implies $\sum_i v_i = 0$, it follows
    that both $S$ and $T$ are non-empty. Since $i \in S$ and $j \in T$ imply that $v_i \ne v_j$, it follows that
    $\gamma^{n-1}(\Sigma_{ij}) = 0$ whenever $i \in S$ and $j \in T$.

    Now let $U = \bigcup_{i \in S} \Omega_i$. By (\ref{eq:nothing-lost}) it follows that $\per(U) = \gamma^{n-1} (\partial^* U) = 0$ and hence $\H^{n-1}(\partial^* U) = 0$. But the condition $\gamma(\Omega) \in \interior \simplex$ implies that $\gamma(U) = \sum_{i \in S} \gamma(\Omega_i) \in (0, 1)$ and hence $\H^n(U) > 0$ and $\H^n(\R^n \setminus U) > 0$, in contradiction to $\H^{n-1}(\partial^* U) = 0$. (To see the contradiction, one may simply endow Euclidean space $\R^n$ with the standard Gaussian measure $\gamma_G$, infer that $\gamma_G(U) \in (0,1)$, and apply the single-bubble Gaussian isoperimetric inequality to conclude that $\gamma_G^{n-1}(\partial^* U) > 0$ and hence $\H^{n-1}(\partial^* U) > 0$).
\end{proof}

\begin{corollary} \label{cor:pos-def}
Denoting $A_{ij} := \gamma^{n-1}(\Sigma_{ij})$, the matrix:
\[
     L_{A} := \sum_{i < j} A_{ij} (e_i - e_j) (e_i-e_j)^T 
\]
is strictly positive-definite on $E$ when $\gamma(\Omega) \in \interior \simplex$. 
\end{corollary}

\begin{lemma}\label{lem:compensators}
    If $\gamma(\Omega) \in \interior \simplex$ then
    there exists a collection of $C_c^\infty$ vector-fields $Y_1, \dots, Y_{k-1}$ with disjoint supports, so that for every $p=1,\ldots,k-1$,  $(Y_p)|_{\cup_{i<j}\Sigma_{ij}}$ is supported in $\Sigma_{ij}$ for some $i<j$, and such the set $\{\delta_{Y_i} V: i = 1, \dots, k-1\}$ spans $E$.
\end{lemma}

\begin{proof}
    For each pair $i, j$ with $\gamma^{n-1}(\Sigma_{ij}) > 0$, choose some $x_{ij} \in \Sigma_{ij}$.
    By Theorem \ref{thm:Almgren}, there exists some $\eps > 0$ such that $B(x_{ij},\eps)$ is disjoint from all
    the interfaces besides $\Sigma_{ij}$ and $\cl(B(x_{ij},\eps) \cap \Sigma_{ij}) \subset \Sigma_{ij}$, where $\cl$ denotes closure. By replacing $\eps$ by $\eps/2$, we can ensure that all of the $B(x_{ij},\eps)$ are pairwise disjoint.  
    Let $f_{ij}$ be a non-negative $C_c^\infty$ function supported in $B(x_{ij},\eps)$
    such that $f(x_{ij}) > 0$, and let $X_{ij}$ be a smooth extension of $f_{ij} \n_{ij}$ that is supported in $B(x_{ij},\eps)$.
    Then $\delta_{X_{ij}} V = \alpha_{ij} (e_i - e_j)$ for some $\alpha_{ij} > 0$.

    By Lemma~\ref{lem:span}, $\{\delta_{X_{ij}} V: \gamma^{n-1}(\Sigma_{ij}) > 0\}$ spans $E$; hence,
    we may choose $Y_1, \dots, Y_{k-1}$ to be an appropriate subset of the $X_{ij}$.
\end{proof}

\begin{lemma}[Stationarity] \label{lem:first-order}
    Let $\Omega$ be a minimizing cluster.  For any admissible vector-field $X$, if $\delta_X V = 0$ then $\delta_X A = 0$.
\end{lemma}

\begin{proof}
   Let $F_t$ denote the flow along $X$, defined as usual by:
     \[
    \frac{d}{dt} F_t = X \circ F_t ~,~ F_0 = \text{Id} .
    \]
    Choose a family $Y_1, \dots, Y_{k-1}$ of vector-fields as in Lemma~\ref{lem:compensators}, having compact and pairwise disjoint supports. 
    Let $\{F_{t,s}\}_{t \in \R , s \in \R^{k-1}}$ be a family of $C^\infty$ diffeomorphisms defined by solving the following system of linear stationary PDEs:
    \begin{align*}
    \pdiff{}{s_i} F_{t,s} &= Y_i \circ F_{t,s} \;\;\; \forall i=1,\ldots,k-1 \\
     F_{t,\vec 0} &= F_t .
    \end{align*}
                                                                  Observe that the above system is indeed integrable since the $Y_i$'s have disjoint supports, and hence the flows they individually generate necessarily commute (cf. the Frobenius Theorem \cite{Lang-ManifoldsBook}). Consequently:
     \begin{equation} \label{eq:concat}
     F_{t,s} = F^{k-1}_{s_{k-1}} \circ \ldots \circ F^{1}_{s_1} \circ F_t ,
     \end{equation}
     where:
     \[
     \frac{d}{ds} F^i_s = Y_i \circ F^i_s ~,~ F^i_0 = \text{Id} \;\;\; \forall i=1,\ldots,k-1 ~, 
     \]
     and all of the usual smoothness and uniform boundedness of all partial derivatives of any fixed order apply to $F_{t,s}$ for  $t \in [-T,T]$, $s \in [-T,T]^{k-1}$, for any fixed $T > 0$. 
     
    Let $V(t, s) = \gamma(F_{t,s}(\Omega))$ and $A(t,s) = \per(F_{t,s}(\Omega))$. By the assumption (\ref{eq:gamma-regular}) and a tedious yet straightforward adaptation of the proof of Lemma \ref{lem:regular} to a concatenation of (partly non-commuting) flows as in (\ref{eq:concat}), it follows that $V$ and $A$ are both $C^\infty$ on $\{(t,s) : t \in (-\epsilon,\epsilon) , s \in (-\epsilon,\epsilon)^{k-1}\}$ for some $\epsilon > 0$.         Clearly:
    \begin{align*}
        \left.\frac{\partial^m V}{(\partial t)^m}\right|_{t=0,s=0} = \delta_X^m V ~,~ &
        \left.\frac{\partial^m A}{(\partial t)^m}\right|_{t=0,s=0} = \delta_X^m A \;\;\;\;\;\forall m =1,2,\ldots\\
        \left.\frac{\partial V}{\partial s_i}\right|_{t=0,s=0} = \delta_{Y_i} V  ~,~ &
        \left.\frac{\partial A}{\partial s_i}\right|_{t=0,s=0} = \delta_{Y_i} A.
    \end{align*}
    Since $\{\delta_{Y_i} V\}_{i=1,\ldots,k-1}$ span $E$, $\{\partial_{s_i} V_j(0, 0)\}_{ji} : \R^{k-1} \to E$ is full rank. By the implicit
    function theorem, there exists a $\delta \in (0,\epsilon)$ and a $C^\infty$ curve $s(t) : (-\delta,\delta) \rightarrow (-\epsilon,\epsilon)^{k-1}$, such that $s(0) = 0$ and $V(t,s(t)) = V(0,0) = \gamma(\Omega)$ for all $|t| < \delta$. Moreover, $\frac{\partial V}{\partial t}(0,0) = \delta_X V = 0$ and the full-rank of $\{\partial_{s_i} V_j(0, 0)\}$ imply that $s'(0) = 0$. From this property and the chain rule,
    \[
        \diffat{A(t,s(t))}{t}{t=0} = \pdiffat{A(t,s)}{t}{t=0,s=0} = \delta_X A.
    \]
    Hence, we conclude that $\delta_X A = 0$; if not, there is some $t \ne 0$ such that the cluster
    $\tilde \Omega = F_{t,s(t)}(\Omega)$ has $\gamma(\tilde \Omega) = \gamma(\Omega)$ and
    $\per(\tilde \Omega) = A(t,s(t)) < A(0,0) = \per(\Omega)$, contradicting the minimality of $\Omega$.
\end{proof}

\subsection{First-Order Conditions}

First-order conditions for an isoperimetric minimizing cluster are well-understood, and points~\ref{it:first-order-constant},~\ref{it:first-order-cyclic} and~\ref{it:first-order-lambda}
below are well-known. We denote by $E^*$ the dual of $E$; as usual, $E^*$ may be identified with $E$, acting by the Euclidean inner product.

\begin{theorem}\label{thm:first-order-conditions-expanded}
  Assume (\ref{eq:gamma-regular}) holds. If $\Omega$ is an isoperimetric minimizing cluster then:
  \begin{enumerate}[(i)]
    \item On each $\Sigma_{ij}$, $H_{ij,\gamma}$ is constant. \label{it:first-order-constant}
    \item For all $\cyclic \in \Trips$, $\displaystyle \sum_{(i, j) \in \cyclic} H_{ij,\gamma} = 0$. \label{it:first-order-cyclic}  \\
    Equivalently, there is a unique $\lambda \in E^*$ such that $H_{ij,\gamma} = \lambda_i - \lambda_j$. 
    \item For every admissible vector-field $X$: \label{it:first-order-div}
      \[
        \sum_{i<j} \int_{\Sigma_{ij}} \div_{\Sigma,\gamma} X^\tang\, d\gamma^{n-1} = 0.
      \]
    \item For every  admissible vector-field $X$: 
  \begin{equation}
    \delta_X A = \sum_{i<j} H_{ij,\gamma} \int_{\Sigma_{ij}} X^{\n_{ij}}\, d\gamma^{n-1} = \inr{\lambda}{\delta_X V} .
        \label{eq:formula-first-variation-of-area}
  \end{equation}   
  \label{it:first-order-lambda}
  \end{enumerate}
\end{theorem} 

\begin{remark} \label{rem:no-higher-regularity}
The third point essentially says that at any point where the interfaces $\Sigma_{ij}$
meet, their boundary normals must sum to zero. In fact, more delicate boundary regularity results are known. In $\R^2$, the interfaces meet at $120^\circ$ angles at a discrete set of points - this was shown by F.~Morgan in \cite{MorganSoapBubblesInR2} building upon the work of Almgren \cite{AlmgrenMemoirs}, and also follows from the results of  J.~Taylor~\cite{Taylor-SoapBubbleRegularityInR3}. 
  The regularity results of Taylor also cover the case of $\R^3$, establishing that 
the interfaces must meet in threes
at $120^\circ$ angles along smooth curves, in turn meeting in fours at equal angles of $\cos^{-1}(-1/3) \simeq 109^\circ$. A generalization of this to $\Real^n$ for $n \geq 4$ has been announced by B.~White \cite{White-SoapBubbleRegularityInRn} and recently proved in \cite{CES-RegularityOfMinimalSurfacesNearCones}. We will not require these much more delicate boundary regularity results in this work. \end{remark}

\begin{proof}[Proof of Theorem~\ref{thm:first-order-conditions-expanded}]
We will sketch parts~\ref{it:first-order-constant} and~\ref{it:first-order-cyclic}, which are well-known, and provide a more detailed explanation of part~\ref{it:first-order-div}; part~\ref{it:first-order-lambda} is then an immediate consequence of the previous ones. 

Since $\H^{n-1}(\partial^* \Omega_i \setminus \cup_{j \neq i} \Sigma_{ij}) = 0$ by (\ref{eq:nothing-lost}), and since $\Sigma_{ij}$ are all smooth for a minimizing cluster $\Omega$ by Theorem \ref{thm:Almgren}, the smoothness assumption in the second part of Proposition \ref{prop:first-variation} is satisfied and we may appeal to the formulas for first variation of volume and perimeter derived there:
\begin{align*}
\delta_X V(\Omega_i) & = \sum_{j \neq i} \int_{\Sigma_{ij}} X^{\n_{ij}} \, d\gamma^{n-1} \;\;\; \forall i=1,\ldots,k ~, \\
\delta_X A(\Omega) & = \sum_{i < j} \int_{\Sigma_{ij}} (H_{\Sigma_{ij},\gamma} X^{\n_{ij}} + \div_{\Sigma_{ij},\gamma} X^\tang) \, d\gamma^{n-1} ~.
\end{align*}
Note that whenever the support of $X^{\tang}|_{\Sigma_{ij}}$ is contained in $\Sigma_{ij}$, we may integrate by parts to obtain:
\begin{equation} \label{eq:tang-zero}
\int_{\Sigma_{ij}} \div_{\Sigma_{ij},\gamma} X^\tang \, d\gamma^{n-1} = 0  . 
\end{equation}
 
To establish part~\ref{it:first-order-constant}, observe that if $H_{ij,\gamma}$ were not the same at two different points $x_1,x_2 \in \Sigma_{ij}$, we could find by Theorem \ref{thm:Almgren} an $\epsilon > 0$ so that $B(x_1,\epsilon)$, $B(x_2,\epsilon)$ and all the other interfaces are disjoint, $\gamma^{n-1}(B(x_p,\eps) \cap \Sigma_{ij}) > 0$, 
and $\cl(B(x_p,\eps) \cap \Sigma_{ij}) \subset \Sigma_{ij}$, for $p=1,2$. We could then construct a smooth vector-field $X$ supported in $B(x_1,\eps) \cup B(x_2,\eps)$ (so that (\ref{eq:tang-zero}) applies) with $\delta_X V(\Omega) = 0$ while $\delta_X A(\Omega) < 0$,  in violation of Lemma~\ref{lem:first-order}. 

To establish part~\ref{it:first-order-cyclic}, if $\sum_{(i, j) \in \cyclic} H_{ij,\gamma}$ were not zero for some $\cyclic \in \Trips$, we could construct a smooth vector-field compactly supported around three points (one in each $\Sigma_{ij}$) that would preserve weighted volume to first order, while decreasing weighted perimeter to first order, exactly as above, again in violation of Lemma~\ref{lem:first-order}. Clearly, this implies the existence of $\lambda \in \R^{k}$ so that $H_{ij,\gamma} = \lambda_i - \lambda_j$; as $\lambda$ is defined uniquely up to an additive constant, we may assume that $\sum_{i=1}^k \lambda_i = 0$, i.e. that $\lambda \in E^*$ (it will soon be clear that $\lambda$ acts on $E$). 

At this point, we have shown that for any admissible vector-field $X$: 
\begin{align}
\notag & \sum_{i<j} \int_{\Sigma_{ij}}  H_{ij,\gamma} X^{\n_{ij}}\, d\gamma^{n-1} = \sum_{i < j} (\lambda_i - \lambda_j) \int_{\Sigma_{ij}} X^{\n_{ij}} \, d\gamma^{n-1} \\
\notag  & = \sum_{i \neq j} \lambda_i \int_{\Sigma_{ij}} X^{\n_{ij}} \, d\gamma^{n-1} = \sum_{i} \lambda_i \sum_{j \neq i} \int_{\Sigma_{ij}} X^{\n_{ij}}\, d\gamma^{n-1} \\
\label{eq:inr-lambda}  &= \sum_{i} \lambda_i \delta_X V_i =  \inr{\lambda}{\delta_X V},
\end{align}
where we have used above that $\int_{\Sigma_{ij}} X^{\n_{ij}}\, d\gamma^{n-1} = - \int_{\Sigma_{ji}} X^{\n_{ji}}\, d\gamma^{n-1}$. 
For any such $X$, there exists by Lemma \ref{lem:compensators} a vector-field $Y = \sum_{i=1}^{k-1} c_i Y_i$ with $\delta_Y V(\Omega) = -\delta_X V(\Omega)$ so that $Y|_{\cup_{i<j} \Sigma_{ij}}$ is supported in $\cup_{i<j} \Sigma_{ij}$. In particular, we may integrate-by-parts and obtain:
 \begin{equation} \label{eq:lambda1}
 \int_{\Sigma_{ij}} \div_{\Sigma_{ij},\gamma} Y^\tang \, d\gamma^{n-1} = 0\;\;\; \forall i < j . 
 \end{equation}
 Since $\delta_{X+Y} V = 0$, it follows by (\ref{eq:inr-lambda}) that:
 \begin{equation} \label{eq:lambda2}
 \sum_{i<j}\int_{\Sigma_{ij}} H_{ij,\gamma}  (X^{\n_{ij}} + Y^{\n_{ij}}) \, d\gamma^{n-1} = \inr{\lambda}{\delta_{X+Y} V} = 0 . 
 \end{equation}
 In addition, Lemma \ref{lem:first-order} implies that:
 \begin{equation} \label{eq:lambda3}
 \delta_{X+Y} A = 0 . 
 \end{equation}
Combining (\ref{eq:lambda1}), (\ref{eq:lambda2}) and (\ref{eq:lambda3}), we obtain:
 \begin{align*}
0 & = \delta_{X + Y} A  \\
  &= \sum_{i<j} \int_{\Sigma_{ij}} \brac{H_{ij,\gamma} (X^{\n_{ij}} + Y^{\n_{ij}}) + \div_{\Sigma_{ij},\gamma} (X^\tang + Y^\tang)} \, d\gamma^{n-1} \\
  & = \sum_{i<j} \div_{\Sigma_{ij},\gamma} X^\tang  \, d\gamma^{n-1} ,
 \end{align*}
 concluding the proof of part~\ref{it:first-order-div}. 
 
 Part~\ref{it:first-order-lambda} follows immediately since by part~\ref{it:first-order-div} and (\ref{eq:inr-lambda}):
 \begin{align*}
 \delta_X A &= \sum_{i<j} \int_{\Sigma_{ij}} \brac{H_{ij,\gamma} X^{\n_{ij}} + \div_{\Sigma_{ij},\gamma} X^\tang}  \, d\gamma^{n-1} \\
 & = \sum_{i<j} \int_{\Sigma_{ij}}   H_{ij,\gamma} X^{\n_{ij}}  \, d\gamma^{n-1} = \inr{\lambda}{\delta_X V} . 
 \end{align*}
  
\end{proof}

\subsection{Stability}

Before concluding this section, we show that a minimizing cluster is necessarily stable.

\begin{definition}[Index Form $Q$] The Index Form $Q$, associated to a cluster satisfying the first-order conditions of Theorem \ref{thm:first-order-conditions-expanded}, is defined as the following quadratic form:
\[
Q(X) := \delta^2_X A - \scalar{\lambda,\delta^2_X V} .
\]
\end{definition}

\begin{lemma}[Stability] \label{lem:stable}
 Assume (\ref{eq:gamma-regular}) holds, and let $\Omega$ be an isoperimetric minimizing cluster.  Then for any admissible vector-field $X$: 
\[
\delta_X V = 0 \;\; \Rightarrow \;\; Q(X) \geq 0 .
\]
\end{lemma}
Again, this is well-known and standard for compactly supported vector-fields in the single-bubble case, and was proved in \cite{DoubleBubbleInR3} assuming higher-order boundary regularity, which we avoid in this work. Consequently, for completeness, we provide a proof. 

\begin{proof}[Proof of Lemma \ref{lem:stable}]
    Let $Y_1, \dots, Y_{k-1}$, $F_{t,s}$, and $s(t)$ be as in the proof of Lemma~\ref{lem:first-order}. Recall that $\delta_X V = 0$ implies that $s'(0)=0$.
    By the chain rule, it follows (using $s'(0)=0$) that:
    \begin{align*}
                \left.\frac{d^2 A(t,s(t))}{(dt)^2} \right|_{t=0}
        & = \left.\frac{\partial^2 A}{(\partial t)^2}\right|_{t=0,s=0}  + \sum_{i=1}^{k-1} s_i''(0) \pdiffat{A}{s_i}{t=0,s=0} \\
        &= \delta_X^2 A + \sum_{i=1}^{k-1} s_i''(0) \delta_{Y_i} A \\
        &= \delta_X^2 A + \sum_{i=1}^{k-1} s_i''(0) \inr{\lambda}{\delta_{Y_i} V},
    \end{align*}
    where the last equality follows from (\ref{eq:formula-first-variation-of-area}). 
    Differentiating the relation $V(0,0) = V(t,s(t))$ twice in $t$ (and using again that $s'(0) = 0$),
    we obtain:
    \[
    0 = \left.\frac{\partial^2 V}{(\partial t)^2}\right|_{t=0,s=0} + \sum_i s_i''(0) \pdiffat{V}{s_i}{t=0,s=0} = \delta_X^2 V + \sum_i s_i''(0) \delta_{Y_i} V .
    \]
    Hence,
    \[
        \left.\frac{d^2 A(t,s(t))}{(dt)^2} \right|_{t=0}
        = \delta_X^2 A - \inr{\lambda}{\delta_X^2 V} = Q(X).
    \]
    It follows that necessarily $Q(X) \ge 0$, since otherwise, recalling that $\left.\frac{dA(t,s(t))}{dt} \right|_{t=0} = 0$ by Lemma~\ref{lem:first-order},
    we would find some $t \ne 0$ so that the cluster $\tilde \Omega = F_{t,s(t)}(\Omega)$ satisfies $\gamma(\tilde \Omega) = \gamma(\Omega)$
    and $\per(\tilde \Omega) = A(t,s(t)) < A(0,0) = \per(\Omega)$, contradicting the minimality of $\Omega$.
\end{proof}

\section{Second Variations} \label{sec:second-order}

Unlike the formulas for first variation above, the following identity for the second
variation under translations appears to be new, and moreover, is particular to the Gaussian measure. 

\begin{theorem}\label{thm:formula}
  If $\Omega$ is an isoperimetric minimizing cluster for the Gaussian measure $\gamma$ then for any $w \in \R^n$,
    \begin{equation}
        Q(w) = \delta_w^2 A - \inr{\lambda}{\delta_w^2 V} = - \sum_{i < j} \int_{\Sigma_{ij}} \inr{w}{\n_{ij}}^2\, d\gamma^{n-1}.
            \label{eq:formula}
    \end{equation}
\end{theorem}

Here  $\delta_w$ denotes the variation for the constant vector-field $X \equiv w \in \R^n$. Recall that given an isoperimetric minimizing cluster, $\lambda \in E^*$ denotes the unique vector guaranteed by Theorem \ref{thm:first-order-conditions-expanded} to satisfy  $\lambda_i - \lambda_j = H_{ij,\gamma}$. We continue to treat the case of a $k$-cell cluster, $k \geq 3$, as this poses no additional effort.

\subsection{Second variation of volume}

\begin{lemma} \label{lem:delta2-V}
Let $\Omega$ be an isoperimetric minimizing cluster for a measure $\gamma = e^{-\pot} dx$ satisfying (\ref{eq:gamma-regular}). Then for any $i=1,\ldots,k$ and $w \in \R^n$:
\begin{equation} \label{eq:delta2-V-cell}
\delta_w^2 V(\Omega_i) = - \sum_{j \neq i}\int_{\Sigma_{ij}} \inr{w}{\n_{ij}} \nabla_w \pot \, d\gamma^{n-1} .
\end{equation}
   In particular:
   \[
    \inr{\lambda}{\delta_w^2 V} = - \sum_{i < j} H_{ij,\gamma} \int_{\Sigma_{ij}} \inr{w}{\n_{ij}} \nabla_w \pot \, d\gamma^{n-1} .
    \]
\end{lemma}
\begin{proof}
By (\ref{eq:gamma-regular}) $\Omega_i$ is volume regular, and so by Lemma \ref{lem:regular} we have:
\begin{align*}
  \delta_w^2 V(\Omega_i)
  &= \int_{\Omega_i} \diffIIat{e^{-\pot(x + tw)}}{t}{t}{t=0}\, dx \\
  &= \int_{\Omega_i} (\nabla_w \pot)^2 - \nabla^2_{w,w} \pot\, d\gamma \\
  & = - \int_{\Omega_i} \div_{\gamma} (w \nabla_w \pot) \, d\gamma . 
\end{align*}
  In order to apply integration-by-parts, we need to first make $X = w \nabla_w \pot$ compactly supported. 
Noting that $\abs{X} \leq \abs{w}^2 \abs{\nabla \pot}$ and $\abs{\div X} = \abs{\nabla^2_{w,w} \pot} \leq \abs{w}^2 \norm{\nabla^2 \pot}$, we may apply Lemma \ref{lem:cutoff} to approximate $X$ by the compactly-supported $\eta_R X$. Integrating by parts using ~\eqref{eq:integration-by-parts}, we obtain:
\begin{align*}
  &= - \lim_{R \rightarrow \infty} \int_{\Omega_i} \div_\gamma(\eta_R w \nabla_w \pot)\, d\gamma \\
  &= -\lim_{R \rightarrow \infty} \int_{\partial^* \Omega_i} \eta_R \inr{w}{\n} \nabla_w \pot \, d\gamma^{n-1} \\
  & = - \int_{\partial^* \Omega_i} \inr{w}{\n} \nabla_w \pot \, d\gamma^{n-1}  ,
\end{align*}
where the last equality follows by Dominant Convergence since $\Omega_i$ is perimeter regular by (\ref{eq:gamma-regular}). Since $\H^{n-1}(\partial^* U_i  \setminus \cup_{j \neq i} \Sigma_{ij}) = 0$, (\ref{eq:delta2-V-cell}) follows. 

Finally, since swapping $i$ and $j$ changes the sign of $\inr{w}{\n_{ij}}$, we have:
\begin{align*} 
    \inr{\lambda}{\delta_w^2 V}
        & = - \sum_{i=1}^k \lambda_i \sum_{j \ne i} \int_{\Sigma_{ij}} \inr{w}{\n_{ij}} \nabla_w \pot \, d\gamma^{n-1} \\    
    &= - \sum_{i < j} (\lambda_i - \lambda_j)\int_{\Sigma_{ij}} \inr{w}{\n_{ij}} \nabla_w \pot \, d\gamma^{n-1} \notag \\
    &= - \sum_{i < j} H_{ij,\gamma} \int_{\Sigma_{ij}} \inr{w}{\n_{ij}} \nabla_w \pot \, d\gamma^{n-1}.
\end{align*}
\end{proof}

\subsection{Second variation of perimeter}

\begin{lemma}
Let $\Omega$ be an isoperimetric minimizing cluster for a measure $\gamma = e^{-\pot} dx$ satisfying (\ref{eq:gamma-regular}). Then for any $w \in \R^n$:
\[
\delta_w^2 A = -  \sum_{i < j} \int_{\Sigma_{ij}} \brac{ H_{ij,\gamma} \scalar{w,\n_{ij}} \nabla_w \pot  + \scalar{w, \n_{ij}} \nabla^2_{\n_{ij},w} \pot } d\gamma^{n-1} .
\]
\end{lemma}
\begin{proof}
By (\ref{eq:gamma-regular}) $\Omega_i$ is perimeter regular, and so by Lemma \ref{lem:regular} we have for every $i$:
\begin{align*}
  \delta_w^2 A(\Omega_i)
&  = \int_{\partial^* \Omega_i} \diffIIat{e^{-\pot(x + tw)}}{t}{t}{t=0}\, d\calH^{n-1}(x) \\
&  = \int_{\partial^* \Omega_i} \brac{(\nabla_w \pot)^2 - \nabla^2_{w,w} \pot} d\gamma^{n-1} \\
&  = \sum_{j \neq i} \int_{\Sigma_{ij}} \brac{(\nabla_w \pot)^2 - \nabla^2_{w,w} \pot} d\gamma^{n-1}  . 
\end{align*}
On the other hand, we now have:
\begin{equation}\label{eq:div-constant}
  \div_{\Sigma_{ij},\gamma} (w \nabla_w \pot)
  = \div_{\Sigma_{ij}} (w \nabla_w \pot) - (\nabla_w \pot)^2
  = \nabla^2_{w^\tang,w} \pot - (\nabla_w \pot)^2,
\end{equation}
where $w^\tang$ is the tangential part of $w$. Consequently, we obtain:
\[
 \delta_w^2 A(\Omega_i) = - \sum_{j \neq i} \int_{\Sigma_{ij}} \brac{\div_{\Sigma_{ij},\gamma} (w \nabla_w \pot) + \nabla^2_{w,w} \pot - \nabla^2_{w^\tang,w} \pot} d\gamma^{n-1} . 
 \]
Summing over $i$ and dividing by $2$, we obtain:
\[
\delta_w^2 A = - \sum_{i < j} \int_{\Sigma_{ij}} \brac{\div_{\Sigma_{ij},\gamma} (w \nabla_w \pot) + \scalar{w, \n_{ij}} \nabla^2_{\n_{ij},w} \pot  } d\gamma^{n-1} . 
\]

We now claim that the contribution of the tangential part of divergence terms vanishes. This is essentially the content of Theorem~\ref{thm:first-order-conditions-expanded} part~\ref{it:first-order-div}, applied to the vector-field $X = w \nabla_w \pot$; however, $X$ is not bounded and so does not satisfy the required assumption (\ref{eq:field-bdd}). To remedy this, note as before that $\abs{X} \leq \abs{w}^2 \abs{\nabla \pot}$ and $\abs{\div_\Sigma X} = \abs{\nabla^2_{w^{\tang},w} \pot} \leq \abs{w}^2 \norm{\nabla^2 \pot}$, and so we may apply Lemma \ref{lem:cutoff} to approximate $X$ by the compactly-supported $\eta_R X$. Now applying Theorem~\ref{thm:first-order-conditions-expanded} part~\ref{it:first-order-div} to $\eta_R X$, we deduce:
\begin{align*}
\delta_w^2 A  & = - \lim_{R \rightarrow \infty} \sum_{i < j} \int_{\Sigma_{ij}} \brac{\div_{\Sigma_{ij},\gamma} (\eta_R w \nabla_w \pot) + \scalar{w, \n_{ij}} \nabla^2_{\n_{ij},w} \pot } d\gamma^{n-1} \\
& =  - \lim_{R \rightarrow \infty} \sum_{i < j} \int_{\Sigma_{ij}} \brac{\div_{\Sigma_{ij},\gamma} (\eta_R \n_{ij} \scalar{w,\n_{ij}} \nabla_w \pot) + \scalar{w, \n_{ij}} \nabla^2_{\n_{ij},w} \pot } d\gamma^{n-1} \\
& =  - \lim_{R \rightarrow \infty} \sum_{i < j} \int_{\Sigma_{ij}} \brac{ \eta_R H_{ij,\gamma} \scalar{w,\n_{ij}} \nabla_w \pot  + \scalar{w, \n_{ij}} \nabla^2_{\n_{ij},w} \pot } d\gamma^{n-1} \\
& = -  \sum_{i < j} \int_{\Sigma_{ij}} \brac{ H_{ij,\gamma} \scalar{w,\n_{ij}} \nabla_w \pot  + \scalar{w, \n_{ij}} \nabla^2_{\n_{ij},w} \pot } d\gamma^{n-1} ,
\end{align*}
where the last equality follows by Dominant Convergence since the cells $\Omega_i$ are perimeter regular by (\ref{eq:gamma-regular}). 
\end{proof}

\begin{corollary} \label{cor:delta2-A}
For $\gamma$ the standard Gaussian measure we have:
\[
\delta_w^2 A  = -  \sum_{i < j} \int_{\Sigma_{ij}} \brac{ H_{ij,\gamma} \scalar{w,\n_{ij}} \nabla_w \pot  +  \inr{w}{\n_{ij}}^2 } d\gamma^{n-1} .
\]
\end{corollary}
\begin{proof}
Apply Corollary \ref{cor:Gaussian-regular} and use $\nabla^2 W = \text{Id}$. 
\end{proof}

\begin{proof}[Proof of Theorem \ref{thm:formula}]
Immediate from combining Lemma \ref{lem:delta2-V} and Corollary \ref{cor:delta2-A}.
\end{proof}

\section{The differential inequality} \label{sec:MDI}

From hereon, we specialize to the case of the standard Gaussian measure $\gamma$ on $\R^n$ and 3-cell clusters (the double-bubble problem).  
In this section, we prove a rigorous version of the inequality $\nabla^2 I \le -L_A^{-1}$ (in view of the fact that we don't yet know
$I$ to be differentiable). 

\begin{theorem}[Main Differential Inequality] \label{thm:hessian-bound-for-I}
    Fix $v \in \interior \simplex$. Let $\Omega$ be an isoperimetric minimizing cluster
    with $\gamma(\Omega) = v$, and let $A_{ij} = \gamma^{n-1}(\Sigma_{ij})$, where
    $\Sigma_{ij}$ are the interfaces of $\Omega$.
    Then for any $y \in E$, there exists an admissible vector-field $X$ such that:
    \[
     \delta_X V = y \;\; \text{ and } \;\; Q(X) \le - y^T L_A^{-1} y ,
     \]
     where:
     \[
      L_{A} := \sum_{i < j} A_{ij} (e_i - e_j) (e_i-e_j)^T .
     \]
\end{theorem}

Recall that by Corollary \ref{cor:pos-def}, $L_A$ is strictly positive-definite on $E$. For the proof, define the ``average'' normals $\navg_{ij}$ by
\[
  \navg_{ij} = \begin{cases} \frac{1}{A_{ij}} \int_{\Sigma_{ij}} \n_{ij}\, d\gamma^{n-1} & A_{ij} > 0 \\ 0 & A_{ij} = 0 \end{cases} . 
\]
By~\eqref{eq:formula-first-variation-of-volume}, we have:
\[
(\delta_w V)_i = \sum_{j \neq i} \int_{\Sigma_{ij}} \inr{w}{\n_{ij}} \, d\gamma^{n-1} = \sum_{j \neq i} A_{ij} \inr{w}{\navg_{ij}} .
\]
Since $\navg_{ji} = -\navg_{ij}$, we can write this as:
\[
    \delta_w V = \sum_{i<j} A_{ij} (e_i - e_j) \navg_{ij}^T w = M w ,
\]
where $M: \R^n \to E$ is the following linear operator:
\begin{equation} \label{eq:M}
    M := \sum_{i<j} A_{ij} (e_i - e_j) \navg_{ij}^T .
\end{equation}

We first observe the following dichotomy: for a minimizing cluster,
either $M$ is surjective or the cluster is ``effectively one-dimensional.''

\begin{definition}
    The cluster $\Omega$ is called effectively one-dimensional if there exist
    a cluster $\tilde \Omega$ in $\R$ and $\theta \in S^{n-1}$ such that
    for all $i$, $\Omega_i$ and $\{x \in \R^n: \inr{x}{\theta} \in \tilde \Omega_i\}$ coincide up to a null-set.
\end{definition}
 
\begin{lemma}\label{lem:dichotomy}
    Let $n \ge 2$ and let $\Omega$ be an isoperimetric minimizing cluster for the Gaussian measure. Then exactly one of the following possibilities holds: either $M: \R^n \to E$ is surjective, or $\Omega$ is effectively one-dimensional.
\end{lemma}

\begin{proof}
   If $\Omega$ is effectively one-dimensional then for all $x \in \Sigma = \cup_{i<j} \Sigma_{ij}$, the normals $\n_{ij}$ coincide with $\pm \theta \in S^{n-1}$, and hence $\navg_{ij}$ all lie in the linear span of $\theta$, so that $M$ is not surjective (as $n \geq 2$). 

    For the converse direction, we may assume that $\max_i \gamma(\Omega_i) < 1$, otherwise there is nothing to prove. Note that for any $w \in \ker M$, $\delta_w V = M w = 0$, and so by stability (Lemma~\ref{lem:stable}) and~\eqref{eq:formula}, 
    \[
    Q(w) = - \sum_{i < j} \int_{\Sigma_{ij}} \inr{w}{\n_{ij}}^2\, d\gamma^{n-1} \geq 0 . 
    \]
   It follows that $\inr{w}{\n_{ij}}$ must vanish $\gamma^{n-1}$-a.e.\ on each interface $\Sigma_{ij}$. Since $w \in \ker M$ was arbitrary, we see that $\n \in (\ker M)^\perp$ a.e. on $\Sigma$. But since $\gamma^{n-1}(\Sigma) \geq \frac 12 \sum_i \isop(\gamma(\Omega_i)) > 0$ (by applying the single-bubble Gaussian isoperimetric profile $\isop: [0, 1] \to\R_+$ to each cell $\Omega_i$), and since $\abs{\n} = 1$, it follows that $M \neq 0$. 
   
   Consequently, if $M: \R^n \to E$ is not surjective then necessarily (since $\dim E = 2$) $\dim \ker M = n-1$. This means that there exists $\theta \in S^{n-1}$ such that $\n$ is a.e.\ equal to $\pm \theta$ on $\Sigma$. By rotational invariance of the Gaussian measure,
    we may assume without loss of generality that $\theta = e_1$.

    Let us construct the one-dimensional cluster $\tilde \Omega$ witnessing that $\Omega$
    is effectively one-dimensional. Fix $i \ne j$; since $\n_{ij}$ is continuous on $\Sigma_{ij}$
    and takes only two values, it must be locally constant; hence, $H_{ij}(x) = 0$ for all $x \in \Sigma_{ij}$,
    and so $H_{ij,\gamma}(x) = -\inr{x}{\n_{ij}} \in \{ \pm x_1 \}$ for all $x \in \Sigma_{ij}$.
    Since $H_{ij,\gamma}$ is constant on $\Sigma_{ij}$, it follows that $\Sigma_{ij}$ is contained
    in the two halfspaces $\{x: x_1 = \pm a\}$, for some $a \in \R$. Hence, $\partial^* \Omega_i$ is
    contained in at most four halfspaces: call them $\{x: x_1 \in \{\pm a, \pm b\}\}$.

    Up to modifying $\Omega_i$ on a $\gamma^n$-null set, $\partial^* \Omega_i$ is dense in $\partial \Omega_i$
    (see~\cite[Proposition~12.19]{MaggiBook}). Hence, $\partial \Omega_i$ is also contained in
    $\{x: x_1 \in \{\pm a, \pm b\}\}$. Now, $\R \setminus \{\pm a, \pm b\}$ has at most five connected
    components. For each of these components $U$, either $U \times \R^{n-1} \subset \Omega_i$
    or $U \times \R^{n-1} \cap \Omega_i = \emptyset$, because anything else would contradict
    the fact that $\partial \Omega_i \subset \{x: x_1 \in \{\pm a, \pm b\}\}$. Now define $\tilde \Omega_i$
    to be the union of those intervals $U$ satisfying $U \times \R^{n-1} \subset \Omega_i$.
    Repeating this construction for all $i$, we see that $\Omega$ is effectively one-dimensional.
\end{proof}

Based on the dichotomy in Lemma~\ref{lem:dichotomy}, we will prove Theorem~\ref{thm:hessian-bound-for-I}
in either of the two cases. Note that \emph{a-posteriori} we will show that the unique isoperimetric minimizing clusters $\Omega$ (with $\gamma(\Omega) \in \interior \Delta$) are tripods, which are clearly \emph{not} effectively one-dimensional, but we do not exclude this possibility beforehand.  

\smallskip
 We will require the following matrix form of the Cauchy--Schwarz inequality:

\begin{lemma}\label{lem:cs}
    If $X$ is a random vector in $\R^n$ and $Y$ is a random vector in $\R^k$ such that
    $\E |X|^2 < \infty$, $\E |Y|^2 < \infty$, and $\E Y Y^T$ is non-singular,
    then
    \[
        (\E XY^T) (\E Y Y^T)^{-1} (\E YX^T) \le \E X X^T
    \]
    (in the positive semi-definite sense)
    with equality if and only if $X = BY$ a.s., for a deterministic matrix $B$.
\end{lemma}

While this inequality may be found in the literature (e.g.~\cite{Tripathi:99}), we present a short independent proof, including that of the equality case, which we will require for establishing uniqueness of minimizing clusters. 

\begin{proof}[Proof of Lemma \ref{lem:cs}]
  Let $Z = X - (\E X Y^T) (\E Y Y^T)^{-1} Y$. Then $Z Z^T \ge 0$ (in the positive semi-definite sense),
  and since convex combinations of positive semi-definite matrices are also positive semi-definite,
  $\E Z Z^T \ge 0$. A computation verifies that:
  \[
    \E Z Z^T = \E X X^T - (\E XY^T) (\E Y Y^T)^{-1} (\E Y X^T) ,
  \]
  completing the proof of the inequality.

  To check the equality case, note that equality occurs iff $\E Z Z^T = 0$, iff $Z = 0$ a.s., iff $X = (\E X Y^T) (\E Y Y^T)^{-1} Y$ a.s. .
\end{proof}

Consider a minimizing cluster $\Omega$ with $\gamma(\Omega) = v \in \interior \simplex$.
By~\eqref{eq:formula} and the Cauchy--Schwarz (or Jensen) inequality:
\[
    Q(w) = \delta_{w}^2 A - \inr{\delta_{w}^2 V}{\lambda} = -\sum_{i<j} \int_{\Sigma_{ij}} \inr{w}{\n_{ij}}^2\, d\gamma^{n-1} \le -\sum_{i<j} \inr{w}{\navg_{ij}}^2 A_{ij} .
\] 
Denoting:
\begin{equation} \label{eq:N}
    N := \sum_{i<j} A_{ij} \navg_{ij} \navg_{ij}^T ,
\end{equation}
we have shown that:
\[
Q(w) \leq - w^T N w \;\;\; \forall w \in \R^n . 
\]

Now choose random vectors $X \in \R^n$ and $Y \in \R^k$ so that with probability $A_{ij}/\sum_{i<j} A_{ij}$, $X = \navg_{ij}$ and $Y = e_i - e_j$. Then, by definition of $N$, $M$ and $L_A$:
\[
\E X X^T = \frac{1}{ \sum_{i<j} A_{ij}} N ~,~ \E Y X^T = \frac{1}{ \sum_{i<j} A_{ij}} M ~,~ E Y Y^T = \frac{1}{ \sum_{i<j} A_{ij}} L_A .
\]
Lemma~\ref{lem:cs} implies that $M^T L_A^{-1} M \le N$, yielding:
\[
Q(w) \le -w^T N w \leq -w^T M^T L_A^{-1} M w \;\;\; \forall w \in \R^n . 
\]
Consequently, when $M$ is surjective, for any $y \in E$, we may choose $w \in \R^n$ so that:
\[
\delta_w V = M w = y \;\; \text{ and } \;\; Q(w) \leq -y L_A^{-1} y , 
\]
completing the proof of Theorem~\ref{thm:hessian-bound-for-I} in that case. 

\medskip

Now suppose that $M$ is not surjective. By Lemma~\ref{lem:dichotomy}, $\Omega$ is effectively one-dimensional.
To lighten our notation, we will assume that $\Omega$ itself is a cluster in $\R$; by the product structure of the Gaussian measure, everything
that we do here can be lifted to the original cluster in $\R^n$ by taking Cartesian product with $\R^{n-1}$. 

Now, the fact that $H_{ij,\gamma}$ is constant implies that each $\Sigma_{ij}$ can contain at most two points.
For each $i \ne j$, define the vector-field $X_{ij}$ by defining $X_{ij} = \n_{ij}$ on $\Sigma_{ij}$,
$X_{ij} = 0$ on all other interfaces, and extending $X_{ij}$ to a $C_c^\infty$
vector-field on $\R$ (it is possible to make it have compact support because there are at most finitely many points
in $\Sigma_{ij}$).
Choose $y \in E$, let $a = L_A^{-1} y$, and let $X = \sum_{i<j} (a_i - a_j) X_{ij}$.
Because the first two derivatives of the
one-dimensional Gaussian density $\varphi$ are $\varphi'(x) = -x\varphi(x)$ and $\varphi''(x) = (x^2 - 1) \varphi(x)$,
we easily compute that 
\begin{align*}
  \delta_X V_i &= \sum_{j \ne i} (a_i - a_j) \sum_{x \in \Sigma_{ij}} \varphi(x) = (L_A a)_i = y_i ~,\\
  \delta_X^2 A &= \sum_{i<j} (a_i - a_j)^2 \sum_{x \in \Sigma_{ij}} (x^2 - 1) \varphi(x) ~,\\
  \delta_X^2 V_i &= \sum_{j \ne i} (a_i - a_j)^2 \sum_{x \in \Sigma_{ij}} \n_{ij}(x) \cdot (-x \varphi(x)) ~. 
\end{align*}
On the other hand, at $x \in \Sigma_{ij}$ one has $H_{ij,\gamma} = \n_{ij}(x) \cdot (-x)$.
Since $\n_{ij} \in \{-1, 1\}$, it follows that
$H_{ij,\gamma} \n_{ij}(x) \cdot (-x) = x^2$, and so:
\begin{align*}
& \inr{\lambda}{\delta_X^2 V} = \sum_{i} \lambda_i \delta_X^2 V_i = \sum_{i<j} (\lambda_i - \lambda_j) (a_i - a_j)^2 \sum_{x \in \Sigma_{ij}} \n_{ij}(x) \cdot (-x \varphi(x)) \\
& = \sum_{i<j} H_{ij,\gamma} (a_i - a_j)^2 \sum_{x \in \Sigma_{ij}} \n_{ij}(x) \cdot (-x \varphi(x)) = \sum_{i<j} (a_i - a_j)^2 \sum_{x \in \Sigma_{ij}} x^2 \varphi(x) .
\end{align*}
Consequently:
\begin{align*}
    Q(X) &= \delta_X^2 A - \inr{\lambda}{\delta_X^2 V} = -\sum_{i<j} (a_i - a_j)^2 \sum_{x \in \Sigma_{ij}} \varphi(x) \\
    &= -\sum_{i<j} (a_i - a_j)^2 A_{ij}  = -a^T L_A a = -y^T L_A^{-1} y.
\end{align*}
This completes the proof of Theorem~\ref{thm:hessian-bound-for-I}.

\begin{remark}
A slight variation on the above argument actually shows that if the minimizing cluster $\Omega$ is effectively one-dimensional (with each cell of $\tilde \Omega$ consisting of finitely many intervals), then up to null-sets, each cell of $\tilde{\Omega}$ must in fact be a single (connected) interval. However, we will show in Section \ref{sec:uniqueness} that minimizing clusters $\Omega$ with $\gamma(\Omega) \in \interior \simplex$ must actually be tripod clusters (up to null-sets), and so in particular they cannot be effectively one-dimensional, and hence there is no point to insist on this additional information here.  
\end{remark}

\section{Proof of the Double-Bubble Theorem} \label{sec:proof}

Recall that the Gaussian double-bubble profile $I$ is defined by
\[
    I(v) = \inf\{\per(\Omega): \Omega \text{ is a 3-cluster with $\gamma(\Omega) = v$}\}.
\]
We are finally ready to prove our main theorem: that the Gaussian double-bubble
profile $I$ agrees with the model (tripod) profile $I_m$. 
We begin with two straightforward observations about $I$. 
\begin{lemma}\label{lem:semi-cont}
  $I : \Delta \to \Real_+$ is lower semi-continuous.
\end{lemma}

\begin{proof} 
The proof is identical to the one of Theorem \ref{thm:Almgren} (i).  If $v_r \to v \in \simplex$, let $\Omega^r$ be a minimizing cluster with $\gamma(\Omega^r) = v_r$.  Note that for all $r$, $P(\Omega^r) = I(v_r) \leq I_m(v_r) \leq \max_{v \in \Delta} I_m(v) =: C < \infty$, and that $P(\Omega^r_i) \leq P(\Omega^r)$ for each $i$. Consequently, by passing to a subsequence, each of the cells $\Omega^r_i$ converges in $L^1(\gamma)$ to $\Omega_i$. By Dominated Convergence (as the total mass is finite), we must have $\gamma(\Omega_i) = v_i$, and 
   the limiting $\Omega$ is easily seen to be a cluster (possibly after a measure-zero modification to ensure disjointness of the cells).  Consequently:
  \[
    I(v) \le \per(\Omega) \le \liminf_{r \to \infty} \per(\Omega^r) = \liminf_{r \to \infty} I(v_r).
  \]
  Hence, $I$ is lower semi-continuous.

\end{proof}

\begin{lemma}\label{lem:agreement-on-boundary}
    On $\partial \simplex$, $I = I_m$.
\end{lemma}
\begin{proof}
    Take $v \in \partial \simplex$. Then some $v_i$ is zero; without loss of generality
    suppose that $v_1 = 0$. The (single-bubble) Gaussian isoperimetric inequality states that
    $I(v) = \isop(v_2) = \isop(v_3)$, where recall $\isop : [0,1] \to \Real_+$ denotes the single-bubble Gaussian isoperimetric profile. But by Lemma~\ref{lem:tripod-profile-continuous}
    this is also equal to $I_m(v)$, concluding the proof. 
\end{proof}

\begin{proof}[Proof of Theorem~\ref{thm:main-I-I_m}]
    By definition, $I \le I_m$, so we will show the other direction.
    Since $I_m$ is continuous and $I$ is lower semi-continuous on $\Delta$, $I - I_m$ attains a global minimum at $v_0 \in \Delta$. Suppose in the contrapositive that $I(v_0) - I_m(v_0) < 0$. By Lemma~\ref{lem:agreement-on-boundary}, we must have $v_0 \in \interior \simplex$.

    Let $\Omega$ be a minimizing cluster with $\gamma(\Omega) = v_0$.
    Since $v_0$ is a minimum point of $I - I_m$, for any smooth flow $F_t$ along a vector-field $X$,
    we have by definition:
    \begin{equation}\label{eq:I_m-and-I}
        I_m(\gamma(F_t(\Omega))) \le I(\gamma(F_t(\Omega)) - (I - I_m)(v_0) \le \per(F_t(\Omega)) - (I - I_m)(v_0) .
    \end{equation}
        The two sides above
    are equal when $t = 0$, and twice differentiable in $t$. By comparing first and second
    derivatives at $t=0$, we conclude that:
    \begin{equation}
        \label{eq:compare-first}
        \inr{\nabla I_m(v_0)}{\delta_X V} = \delta_X A ~,
    \end{equation}
    \begin{equation}
        \label{eq:compare-second}
        \nabla^2_{\delta_X V,\delta_X V} I_m(v_0) + \inr{\nabla I_m(v_0)}{\delta_X^2 V} \le \delta_X^2 A.
    \end{equation}

    Now recall that by (\ref{eq:formula-first-variation-of-area}), we also have $\inr{\lambda}{\delta_X V} = \delta_X A$ for all admissible vector-fields $X$, 
    where as usual  $\lambda$ is the unique element of $E^*$ such that $\lambda_i - \lambda_j = H_{ij,\gamma}(\Omega)$. Since for every $y \in E$ there is a vector-field $X$ as above with $\delta_X V = y$ (e.g. by Theorem~\ref{thm:hessian-bound-for-I}), it follows from~\eqref{eq:compare-first} that necessarily $\nabla I_m(v_0) = \lambda$. 
    
    Applying this to~\eqref{eq:compare-second}, we deduce:
    \[
        \nabla^2_{\delta_X V,\delta_X V} I_m(v_0) \le \delta_X^2 A - \inr{\delta_X^2 V}{\lambda} = Q(X).
    \]
    By Theorem~\ref{thm:hessian-bound-for-I}, for any $y \in E$ we may choose $X$ so that:
    \[
        \nabla^2_{y,y} I_m(v_0) \le Q(X) \le - y^T L_A^{-1} y
    \]
    (where recall $A = A(\Omega) = \set{A_{ij}}$ is the collection of interface measures).
    It follows that $\nabla^2 I_m(v_0) \le - L_A^{-1}$ in the positive semi-definite sense, and since $L_A$ is positive-definite, we deduce that $-\nabla^2 I_m(v_0)^{-1} \le L_A$. It follows by Proposition~\ref{prop:I_m-equation} that:
    \[
        2 I_m(v_0) = -\tr[(\nabla^2 I_m(v_0))^{-1}] \le \tr(L_A) = 2 \sum_{i<j} A_{ij} = 2 I(v_0) ,
    \]
    in contradiction to the assumption that $I(v_0) < I_m(v_0)$.
\end{proof}

\section{Uniqueness of Isoperimetric Minimizers} \label{sec:uniqueness}

In Section \ref{sec:model} we have constructed tripod clusters on $E \simeq
\R^2$. By taking Cartesian product with $\R^{n-2}$ in arbitrary orientation, we obtain tripod clusters in $\R^n$ ($n \geq 2$). Namely, we say that $\Omega$ is a
tripod cluster in $\Real^n$ if there exist three unit-vectors $\{t_i \in \R^n: 1 \le i \le 3\}$ summing to zero, and $x_0 \in \R^n$, 
such that up to $\gamma^n$-null sets,
\[
 \Omega_i = \interior \{x \in \R^n: \max_j \{\inr{x - x_0}{t_j} \} = \inr{x - x_0}{t_i}  \}.
\]
Equivalently, $\Omega$ is a tripod cluster if there exist three unit-vectors $\{n_{ij} \in \R^n:  (i,j) \in \cyclic \}$ summing to zero, and three scalars $\{ h_{ij} \in \R : (i,j) \in \cyclic \}$ summing to zero, so that denoting $n_{ji} := -n_{ij}$ and $h_{ji} := - h_{ij}$, we have up to $\gamma^n$-null sets
\[
 \Omega_i = \bigcap_{j \neq i} \{ x \in \R^n :   \inr{x}{n_{ij}} < h_{ij}  \} .
 \]
 Indeed, the equivalence is easily seen by using $n_{ij} = (t_j - t_i) / \sqrt{3}$ and $h_{ij} = \inr{x_0}{n_{ij}}$. 

\begin{theorem}\label{thm:unique-minimizer}
    If $\Omega$ is an isoperimetric minimizing cluster with $\gamma(\Omega) \in \interior \simplex$
    then $\Omega$ is a tripod cluster.
\end{theorem}

Now that we know Theorem~\ref{thm:main-I-I_m}, the idea is to revisit the various inequalities
used in that proof, and observe that they must be sharp in the case of minimizing clusters.
We begin by restating the relationship between variations and the derivatives of $I$.
We essentially already used this relationship in the proof of Theorem~\ref{thm:main-I-I_m},
but this is easier to state now that we know that $I = I_m$ is smooth on $\interior \Delta$. 

\begin{lemma}\label{lem:minimizer-variation}
  Let $\Omega$ be an isoperimetric minimizing cluster with $\gamma(\Omega) = v \in \interior \simplex$, 
  and let $\lambda \in E^*$ be given by Theorem \ref{thm:first-order-conditions-expanded}, so that $\lambda_i - \lambda_j = H_{ij,\gamma}$, where $H_{ij,\gamma}$ are the weighted mean-curvatures of $\Omega$'s interfaces. Then $\nabla I(v) = \lambda$, and for any admissible vector-field $X$ we have:
  \begin{equation} \label{eq:minimizer-variation-conc}
  (\delta_X V)^T \nabla^2 I(v) \delta_X V \le Q(X) = \delta^2_X A - \inr{\lambda}{\delta^2_X V} . 
  \end{equation}
\end{lemma}

\begin{proof}
  If $\set{F_t}_{t \in \R}$ is a flow along $X$, the definition of $I$ implies that
  \begin{equation}\label{eq:I-inequality-pointwise}
    I(\gamma(F_t(\Omega))) \le \per(F_t(\Omega)),
  \end{equation}
  with equality at $t=0$ because $\Omega$ is minimizing. It follows that the (smooth) functions on either side above must be tangent at $t=0$, and so equating derivatives at $t=0$ we deduce:
      \[
    \inr{\nabla I(v)}{\delta_X V} = \delta_X A,
  \]
    for all admissible vector-fields $X$. Since for every $y \in E$ there is a vector-field $X$ as above with $\delta_X V = y$ (e.g. by Theorem~\ref{thm:hessian-bound-for-I}), it follows from~\eqref{eq:compare-first} that necessarily $\nabla I (v) = \lambda$.
  
  In addition, we must have domination between the second derivatives at $t=0$ of the two sides in~\eqref{eq:I-inequality-pointwise}, and so differentiating twice at $t=0$ we obtain:
   \[
   \inr{\nabla^2 I(v) \delta_X V}{\delta_X V} + \inr{\nabla I(v)}{\delta_X^2 V}\leq \delta_X^2 A . 
   \]
   Rearranging terms and using $\nabla I(v) = \lambda$, (\ref{eq:minimizer-variation-conc}) follows. 
\end{proof}

We immediately deduce from Lemma~\ref{lem:minimizer-variation} and Theorem~\ref{thm:hessian-bound-for-I} that $\nabla^2 I(v) \le - L_A^{-1}$ in the positive-definite sense.
We now observe that this in fact must be an equality. Note that this characterizes the boundary measures $A_{ij}$ for a minimizing cluster.
\begin{lemma}\label{lem:minimizer}
    Let $\Omega$ be an isoperimetric minimizing cluster with $\gamma(\Omega) = v \in \interior \simplex$,
    and let $A_{ij} = \gamma^{n-1}(\Sigma_{ij})$, where $\Sigma_{ij}$ are the interfaces of $\Omega$. Let $A^m_{ij} = \gamma^{n-1}(\Sigma^m_{ij})$, where $\Sigma^m_{ij}$ are the interfaces of a model tripod cluster $\Omega^m$ with $\gamma(\Omega^m) = v$. Then:
    \begin{enumerate}
    	\item $\nabla^2 I(v) = - L_A^{-1}$. 
    	\item $A_{ij} = A^m_{ij}$ for all $i \neq j$. 
    \end{enumerate}
\end{lemma}
\begin{proof}
    By the preceding comments, we know that $\nabla^2 I(v) \le - L_A^{-1}$, and hence $-(\nabla^2 I(v))^{-1} \leq L_A$. On the other hand, by Proposition~\ref{prop:I_m-equation}:
    \[
    \tr[-(\nabla^2 I(v))^{-1}] = \tr[-(\nabla^2 I_m(v))^{-1}] = \tr(L_{A^m}) = 2 I_m(v) = 2 I(v) = \tr(L_A) . 
    \]
    Since $A \geq 0$ and $\tr(A) = 0$ imply that $A=0$, it follows that necessarily $-(\nabla^2 I(v))^{-1} = L_A$, thereby concluding the proof of the first assertion. The second follows by inspecting the off-diagonal elements of the matrices on either side of the following equality:
    \[
    L_A = - (\nabla^2 I(v))^{-1} = - (\nabla^2 I_m(v))^{-1} = L_{A^m} .
    \]
     \end{proof}

\begin{corollary}\label{cor:minimizer}
Under the same assumptions and notation as in Lemma \ref{lem:minimizer}, $A_{ij} > 0$ for all $i \neq j$, and for any admissible vector-field $X$: 
    \[
               - (\delta_X V)^T L_A^{-1} (\delta_X V) \le Q(X) . 
    \]
\end{corollary}

\begin{proof}
    Immediate from the previous lemma since $A^m_{ij} > 0$ and by (\ref{eq:minimizer-variation-conc}). 
\end{proof}

Let us now complete the proof of Theorem~\ref{thm:unique-minimizer}.
Recall the definitions (\ref{eq:M}) and (\ref{eq:N}) of $M$ and $N$ from the proof of Theorem~\ref{thm:hessian-bound-for-I}, and recall that $M w = \delta_w V$. 
In the proof of Theorem~\ref{thm:hessian-bound-for-I},
we proved that for all $w \in \R^n$:
\begin{align}
  Q(w) & = \delta^2_{w} A - \inr{\lambda}{\delta^2_{w} V} = -\sum_{i<j} \int_{\Sigma_{ij}} \inr{w}{\n}^2\, d\gamma^{n-1} \notag \\
  &\le -\sum_{i<j} \inr{w}{\navg_{ij}}^2 A_{ij} = - w^T N w \le -(M w)^T L_{A}^{-1} (M w) . \label{eq:steps-in-inequality}
\end{align}
On the other hand, by Corollary~\ref{cor:minimizer} applied to $X \equiv w$:
\[
-(M w)^T L_{A}^{-1} (M w)  \le Q(w)  .
\]
It follows that all three inequalities above must be equalities for every $w \in \R^n$.
From the first inequality in (\ref{eq:steps-in-inequality}), we conclude that
on each $\Sigma_{ij}$, $\n_{ij}$ is $\gamma^{n-1}$-a.e.\ constant; as $\Sigma_{ij}$ are smooth $(n-1)$-dimensional submanifolds, it follows that $\n_{ij}$ is constant. Since $A_{ij} = \gamma^{n-1}(\Sigma_{ij}) > 0$ by Corollary ~\ref{cor:minimizer}, it follows that $\navg_{ij}$ are non-zero and hence unit-vectors for all $i \neq j$. From the second inequality and the characterization of equality cases in Lemma~\ref{lem:cs},
we deduce the existence of an $n \times k$ matrix $B$ such that $B(e_i - e_j) = \navg_{ij}$
for every $i \ne j$ (using again that $A_{ij} > 0$ for all $i \neq j$).
In particular, this implies that $\sum_{(i,j) \in \cyclic} \navg_{ij} = B \sum_{(i,j) \in \cyclic} (e_i - e_j) =  0$.  

Since the normal vectors $\n_{ij}$ are constant, the unweighted mean curvatures $H_{ij}$ of $\Sigma_{ij}$
vanish, and so $H_{ij,\gamma}(x) = -\inr{x}{\n_{ij}(x)} = -\inr{x}{\navg_{ij}}$ for all $x \in \Sigma_{ij}$.
Recalling that each $H_{ij,\gamma}$ is constant, we must have
\begin{equation}\label{eq:interface-containment}
    \Sigma_{ij} \subset \{x \in \R^n: \inr{x}{\navg_{ij}} = -H_{ij,\gamma}\}.
\end{equation}
Denoting $S_{ij} := \{x \in \R^n: \inr{x}{\navg_{ij}} < -H_{ij,\gamma} \}$, we have $\Sigma_{ij} \subset \partial S_{ij}$.
We will show that up to a $\gamma^n$-null set, $\Omega_i = \bigcap_{j \ne i} S_{ij}$; as $\sum_{(i,j) \in \cyclic} \navg_{ij} = 0$ and $\sum_{(i,j) \in \cyclic} H_{ij,\gamma} = 0$ by Theorem~\ref{thm:first-order-conditions-expanded}, this will establish that $\Omega$ is a tripod cluster and conclude the proof of Theorem~\ref{thm:unique-minimizer}.

To this end, fix $i$, and consider the open set $S_{ab} \cap S_{cd}$, for all 4 possibilities when $(a,b) \in \set{(i,j) , (j,i)}$ and $(c,d) \in \set{(i,k), (k,i)}$, where $i,j,k$ are all distinct. Note that the relative perimeter $P(\Omega_i ; S_{ab} \cap S_{cd})$ is zero since $\H^{n-1}(\partial^* \Omega_i \setminus \cup_{j \neq i} \Sigma_{ij}) = 0$ and 
$\Sigma_{ij} \subset \partial S_{ij}$.  
As $S_{ab} \cap S_{cd}$ is open and connected, it follows by \cite[Exercise 12.17]{MaggiBook} that we may modify $\Omega_i$ by a $\gamma^{n}$-null set, so that $\Omega_i \cap S_{ab} \cap S_{cd}$ is either the empty set or $S_{ab} \cap S_{cd}$, for all 4 possibilities above. We now claim that 
$\Omega_i \cap S_{ij} \cap S_{ik} = S_{ij} \cap S_{ik}$, since otherwise, in all remaining 8 possibilities, we would have $\n_{ij} = -\navg_{ij}$ or $\n_{ik} = -\navg_{ik}$, in violation of $A_{ij} , A_{ik} > 0$. We therefore deduce that $\Omega_i \supset \bigcap_{j \ne i} S_{ij}$ for every $i$. Since $\gamma^n(\R^n \setminus \cup_{i} \Omega_i) = 0$, we conclude that $\Omega_i = \bigcap_{j \ne i} S_{ij}$ up to a $\gamma^n$-null set.

\section{Concluding Remarks} \label{sec:conclude}

\subsection{Extension to measures having strongly convex potentials}

Theorem \ref{thm:main-I-I_m} may be immediately extended to probability measures having strongly convex potentials.

\begin{definition}
A probability measure $\mu$ on $\R^n$ is said to have a $K$-strongly convex potential, $K > 0$, if 
$\mu = \exp(-W(x)) dx$ with $W \in C^\infty$ and $\text{Hess} W \geq K \cdot \Id$. 
\end{definition}

\begin{theorem} \label{thm:CaffarelliCor}
Let $\mu$ be a probability measure having $K$-strongly convex potential. 
Denote by $I_\mu : \Delta \rightarrow \Real_+$ its associated $3$-cluster isoperimetric profile, given by:
\[
I_\mu(v) := \inf \set{P_\mu(\Omega) : \text{$\Omega$ is a $3$-cluster with $\mu(\Omega) = v$}} .
\]
Then:
\[
I_\mu \geq \sqrt{K} I_m \text{ on } \Delta . 
\]
\end{theorem}

The proof is an immediate consequence of the following remarkable theorem by L.~Caffarelli \cite{CaffarelliContraction} (see also \cite{KimEMilmanGeneralizedCaffarelli} for an alternative proof and an extension to a more general scenario):

\begin{theorem}[Caffarelli's Contraction Theorem]
If $\mu$ is a probability measure having a $K$-strongly convex potential,
there exists a $C^{\infty}$ diffeomorphism $T : \R^n \rightarrow \R^n$ which pushes forward the Gaussian measure $\gamma$ onto $\mu$ which is $\frac{1}{\sqrt{K}}$-Lipschitz,  i.e. $\abs{T(x) - T(y)} \leq \frac{1}{\sqrt{K}} \abs{x-y}$ for all $x,y \in \R^n$. 
\end{theorem}

We will require the following calculation:
\begin{lemma}
Let $T : \R^n \to \R^n$ denote a $C^{\infty}$ diffeomorphism pushing-forward $\mu_1 = \Psi_1(x) dx$ onto $\mu_2 = \Psi_2(y) dy$, where $\Psi_1,\Psi_2$ are strictly-positive $C^{\infty}$ densities. 
Let $X$ denote a $C^\infty$ vector-field $X$ on $\R^n$, and let $Y$ denote the vector-field obtained by push-forwarding $X$ via the differential $dT$, $Y = (dT)_* X$, given by:
\[
Y(y) = (dT)_* X(y) := (dT \cdot X)(T^{-1} y) .
\]
Then $Y$ is $C^{\infty}$-smooth and satisfies:
\[
(\div_{\mu_2} Y)(T x) = (\div_{\mu_1} X) (x) \;\;\; \forall x \in \R^n . 
\]
\end{lemma}
\noindent
The proof, which we leave as an exercise, is a simple application of the change-of-variables formula:
\[
\frac{\Psi_1(x)}{\Psi_2(T x)} = \text{det} \; dT (x) . 
\]

\begin{proof}[Proof of Theorem \ref{thm:CaffarelliCor}]
Let $T$ be the map from Caffarelli's theorem, pushing forward $\gamma$ onto $\mu$.  For any $C^{\infty}$ compactly-supported vector-field $X$ with $\abs{X} \leq 1$, denote $Y := (dT)_* X$, and observe that since $T$ is $\frac{1}{\sqrt{K}}$-Lipschitz, then $Y$ is also $C^{\infty}$, compactly-supported, and satisfies $\abs{Y} \leq \frac{1}{\sqrt{K}}$. In addition, for any Borel subset $U$ in $\R^n$, observe that:
\[
\int_U \div_{\mu} Y (y) d\mu(y) = \int_{T^{-1} U}  ( \div_{\mu} Y) (Tx) d\gamma(x) = \int_{T^{-1} U}  \div_{\gamma^n} X (x) d\gamma(x) .
\]
Consequently, taking supremum over all such vector-fields $X$, we deduce that:
\[
\frac{1}{\sqrt{K}} P_{\mu}(U) \geq P_{\gamma}(T^{-1} U) . 
\]
On the other hand, by definition, $\mu(U) = \gamma(T^{-1} U)$. Applying these observations to the cells of an arbitrary $3$-cluster $\Omega$, and applying Theorem \ref{thm:main-I-I_m}, we deduce:
\[
 \frac{1}{\sqrt{K}} P_{\mu}(\Omega) \geq P_{\gamma}(T^{-1} \Omega) \geq I_m(\gamma(T^{-1}(\Omega))) = I_m(\mu(\Omega)) . 
\]
It follows that $I_\mu \geq \sqrt{K} I_m$, as asserted. 
\end{proof}

\begin{remark}
It is possible to extend the above argument to measures $\mu$ with \emph{non-smooth} densities having $K$-strongly convex potentials, namely $\mu = \exp(-W(x)) dx$ with $U(x) = W(x) - \frac{K}{2} \abs{x}^2 : \R^n \rightarrow \R \cup \{+\infty\}$ being convex, 
but we do not insist on this generality here. 
\end{remark}

\subsection{Functional Version}

It is also possible to obtain a functional version of the Gaussian double-bubble isoperimetric inequality for locally Lipschitz functions $f : \R^n \rightarrow \Delta^{(2)}$, in the spirit of Bobkov's functional version of the classical single-bubble one \cite{BobkovGaussianIsopInqViaCube}. 
This and additionally related analytic directions will be described in \cite{EMilmanNeeman-FunctionalVersions}.

\subsection{Gaussian Multi-Bubble Conjecture}

It is also possible to extend our results and to verify the Gaussian Multi-Bubble Conjecture, described in the Introduction, in full generality for all $k \leq n+1$.

\smallskip

To establish the conjecture when $k \geq 4$, we require additional regularity and structural information on the boundary of the interfaces $\Sigma_{ij}$ in $\Real^m$ for $m \leq k-2$. As described in Remark \ref{rem:no-higher-regularity}, such results have been obtained by various authors: F.~Morgan when $m=2$ \cite{MorganSoapBubblesInR2}, J.~Taylor when $m=2,3$ \cite{Taylor-SoapBubbleRegularityInR3}, and B.~White \cite{White-SoapBubbleRegularityInRn} and very recently M.~Colombo, N.~Edelen and L.~Spolaor \cite{CES-RegularityOfMinimalSurfacesNearCones} when $m \geq 4$. We use these results to construct variations of the isoperimetric minimizers which are more complicated than merely translations. As the argument is much more involved and requires several additional ingredients on top of the ones already introduced in this work, we will develop the proof in a subsequent work \cite{EMilmanNeeman-GaussianMultiBubbleConj}. 

\subsection{Stable stationary clusters have flat interfaces}

As an intermediate step in our proof of the Gaussian Multi-Bubble Conjecture for general $k \leq n+1$, we will also show in \cite{EMilmanNeeman-GaussianMultiBubbleConj} that clusters satisfying the aforementioned regularity assumptions, which are isoperimetrically stationary:
\[
\delta_X V = 0 \;\; \Rightarrow \;\; \delta_X A = 0 , 
\]
and stable:
\[
\delta_X V = 0 \;\; \Rightarrow \;\; Q(X) \geq 0 , 
\]
necessarily have flat interfaces. In particular, this applies to the case $k=3$ and $n \geq 2$ treated in this work, and adds additional information to the structure of clusters which are not necessarily isoperimetrically minimizing. In the single-bubble case ($k=2$), this was previously shown by McGonagle and Ross in \cite{McGonagleRoss:15}.

\bibliographystyle{plain}

\def\cprime{$'$} \def\textasciitilde{$\sim$}

\end{document}